\documentclass[reqno,10pt]{amsart}
\usepackage{amssymb}
\usepackage[usenames, dvipsnames]{color}
\usepackage{esint}
\usepackage{hyperref}
\usepackage{todonotes}
\usepackage{verbatim}

\usepackage{cite}
\usepackage[nobysame]{amsrefs}
\def\MR#1{}

\theoremstyle{plain}
\newtheorem{theorem}{Theorem}[section]
\newtheorem{lemma}[theorem]{Lemma}
\newtheorem{corollary}[theorem]{Corollary}
\newtheorem{proposition}[theorem]{Proposition}

\theoremstyle{definition}
\newtheorem{definition}[theorem]{Definition}

\theoremstyle{remark}
\newtheorem{remark}[theorem]{Remark}

\newcommand{\dv}{\operatorname{div}}
\newcommand{\osc}{\operatorname{osc}}

\numberwithin{equation}{section}

\newcommand{\bN}{\mathbb{N}}

\newcommand{\bR}{\mathbb{R}}

\newcommand{\bZ}{\mathbb{Z}}
\newcommand{\bS}{\mathbb{S}}

\newcommand\cR{\mathcal{R}}

\makeatletter
\def\dashint{\operatorname%
{\,\,\text{\bf--}\kern-.98em\DOTSI\intop\ilimits@\!\!}}
\makeatother

\begin{document}
\title[insulated conductivity problem: the non-umbilical case]{Gradient estimates for the insulated conductivity problem: the non-umbilical case}
\author[H. Dong]{Hongjie Dong}
\author[Y.Y. Li]{YanYan Li}
\author[Z. Yang]{Zhuolun Yang}
\address[H. Dong]{Division of Applied Mathematics, Brown University, 182 George Street, Providence, RI 02912, USA}
\email{Hongjie\_Dong@brown.edu}

\address[Y.Y. Li]{Department of
Mathematics, Rutgers University, 110 Frelinghuysen Rd, Piscataway,
NJ 08854, USA}
\email{yyli@math.rutgers.edu}

\address[Z. Yang]{Institute for Computational and Experimental Research in Mathematics, Brown University, 121 South Main Street, Providence, RI 02903, USA}
\email{zhuolun\_yang@brown.edu}

\thanks{H. Dong is partially supported by Simons Fellows Award 007638, NSF Grant DMS-2055244, and the Charles Simonyi Endowment at the Institute of Advanced Study.}
\thanks{Y.Y. Li is partially supported by NSF Grants DMS-1501004, DMS-2000261, and Simons Fellows Award 677077.}
\thanks{Z. Yang is partially supported by the Simons Bridge Postdoctoral Fellowship at Brown University}


\begin{abstract}
We study the insulated conductivity problem with inclusions embedded in a bounded domain in $\bR^n$, for $n \ge 3$. The gradient of solutions may blow up as $\varepsilon$, the distance between the inclusions, approaches to $0$. We established in a recent paper optimal gradient estimates for a class of inclusions including balls. In this paper, we prove such gradient estimates for general strictly convex inclusions. Unlike the perfect conductivity problem, the estimates depend on the principal curvatures of the inclusions, and we show that these estimates are characterized by the first non-zero eigenvalue of a divergence form elliptic operator on $\bS^{n-2}$.
\end{abstract}
\maketitle

\section{Introduction and main results}

In this paper, a continuation of \cite{DLY}, we establish gradient estimates for the insulated conductivity problem in the presence of multiple closely located inclusions in a bounded domain in $\bR^n$, $n \ge 3$. Let $\Omega \subset \bR^n$ be a bounded domain with $C^{2}$ boundary containing two $C^{2,\gamma}~~(0 < \gamma < 1)$ relatively strictly convex open sets $D_{1}$ and $D_{2}$ with dist$(D_1 \cup D_2, \partial \Omega) > c > 0$. Denote $\widetilde{\Omega} := \Omega \setminus \overline{(D_1 \cup D_2)}$. The conductivity problem can be modeled by the following elliptic equation:
\begin{equation*}
\begin{cases}
\mathrm{div}\Big(a_{k}(x)\nabla{u}_{k}\Big)=0&\mbox{in}~\Omega,\\
u_{k}=\varphi(x)&\mbox{on}~\partial\Omega,
\end{cases}
\end{equation*}
where $a_{k}$ denotes the conductivity distribution, that is,
$$
a_k = k \chi_{D_1 \cup D_2} + \chi_{\widetilde{\Omega}}.
$$
Let
$$\varepsilon: = \mbox{dist}(D_1, D_2)$$
be small. When $k$ is large or close to $0$, the gradient of solutions may blow up, and it is significant to capture this singular behavior from an engineering point of view. The problem is motivated by the study of damage and fracture analysis of composite materials in the work of Babu\v{s}ka, Andersson, Smith, and Levin \cite{BASL}, where they studied the Lam\'e system and analyzed numerically that, when the ellipticity constants are bounded away from $0$ and infinity, the gradient of solutions remains bounded independent of the distance between inclusions. Bonnetier and Vogelius \cite{BV} proved it in the context of conductivity problem when inclusions are two touching balls in $\bR^2$. This result was extended by Li and Vogelius \cite{LV} to general second order elliptic equations of divergence form with piecewise H\"older coefficients and general shape of inclusions $D_1$ and $D_2$ in any dimension, and then by Li and Nirenberg \cite{LN} for general second order elliptic systems of divergence form, including the linear system of elasticity. Some higher order derivative estimates in dimension $n=2$ were obtained in \cites{DZ,DL,JiKang}.

When $k$ degenerates to $\infty$ (inclusions are perfect conductors) or $0$ (insulators), it was shown in \cites{Kel, BudCar, Mar} that the gradient of solutions generally becomes unbounded as $\varepsilon \to 0$. For the perfect conductivity problem, it was known that
\begin{equation*}
\begin{cases}
\| \nabla u \|_{L^\infty(\widetilde{\Omega})} \le C\varepsilon^{-1/2} \|\varphi\|_{C^2(\partial \Omega)} &\mbox{when}~n=2,\\
\| \nabla u \|_{L^\infty(\widetilde{\Omega})} \le C|\varepsilon \ln \varepsilon|^{-1} \|\varphi\|_{C^2(\partial \Omega)} &\mbox{when}~n=3,\\
\| \nabla u \|_{L^\infty(\widetilde{\Omega})} \le C\varepsilon^{-1} \|\varphi\|_{C^2(\partial \Omega)} &\mbox{when}~n\ge 4,
\end{cases}
\end{equation*}
see \cites{AKLLL,AKL,Y1,Y2,BLY1,BLY2}. These bounds were shown to be optimal and they are independent of the shape of inclusions, as long as the inclusions are relatively strictly convex. Moreover, many works have been done in characterizing the singular behavior of $\nabla u$, which are significant in practical applications. For further works on the perfect conductivity problem and closely related works, see e.g. \cites{ACKLY,BT1,BT2,DL,AKKY,KLY1,KLY2,L,LLY,LWX,BLL,BLL2,DZ,KL,CY,ADY,Gor,LimYun} and the references therein.

For the insulated conductivity problem, it was proved in \cite{BLY2} that
\begin{equation}\label{insulated_upper_bound}
\| \nabla u \|_{L^\infty(\widetilde{\Omega})} \le C\varepsilon^{-1/2} \|\varphi\|_{C^2(\partial \Omega)}\quad \mbox{when}~n\ge 2.
\end{equation}
The upper bound is optimal for $n = 2$. Yun \cite{Y3} studied the following free space insulated conductivity problem in $\bR^3$: Let $H$ be a harmonic function in $\bR^3$, $D_1 = B_1\left(0,0,1+\frac{\varepsilon}{2} \right)$, and $D_2 = B_1\left(0,0,-1-\frac{\varepsilon}{2} \right)$,
\begin{equation*}
\begin{cases}
\Delta{u}=0&\mbox{in}~\bR^3\setminus\overline{(D_1 \cup D_2)},\\
\frac{\partial u}{\partial \nu} = 0 &\mbox{on}~\partial{D}_{i},~i=1,2,\\
u(x)-H(x) = O(|x|^{-2})&\mbox{as}~|x| \to \infty.
\end{cases}
\end{equation*}
He proved that for some positive constant $C$ independent of $\varepsilon$,
\begin{equation*}
\max_{|x_3|\le \varepsilon/2}|\nabla u(0,0,x_3)| \le C \varepsilon^{\frac{\sqrt{2}-2}{2}}.
\end{equation*}
He also showed that this upper bound of $|\nabla u|$ on the $\varepsilon$-segment connecting $D_1$ and $D_2$ is optimal for $H(x) \equiv x_1$. However, this result does not provide an upper bound of $|\nabla u|$ in the complement of the $\varepsilon$-segment. The upper bound \eqref{insulated_upper_bound} was improved by Li and Yang \cite{LY2} to
\begin{equation*}
\| \nabla u \|_{L^\infty(\widetilde{\Omega})} \le C\varepsilon^{-1/2 + \beta} \|\varphi\|_{C^2(\partial \Omega)}\quad \mbox{when}~n\ge 3,
\end{equation*}
for some $\beta > 0$. When insulators are unit balls, a more explicit constant $\beta(n)$ was given by Weinkove in \cite{We} for $n \ge 4$ by a different method. The constant $\beta(n)$ obtained in \cite{We} presumably improves that in \cite{LY2}. In particular, it was proved in \cite{We} that $\beta(n)$ approaches $1/2$ from below as $n \to \infty$. Despite the significant progress on the conductivity problem that has been made in the past three decades or so, the optimal blow-up rate for the insulated conductivity problem in dimensions $n \ge 3$ remains unknown, and it is described as an outstanding open problem in \cite{Kang}.

In \cite{DLY}, we established optimal gradient estimates for a certain class of inclusions including two balls of any size in dimensions $n \ge 3$. In this paper, we study the insulated conductivity problem with $C^{\gamma}$ coefficients in dimensions $n \ge 3$, with any $C^{2,\gamma}$ relatively strictly convex inclusions:
\begin{equation}\label{general_equation}
\left\{
\begin{aligned}
-\partial_i (A^{ij}(x)  \partial_j u(x)) &=0 \quad \mbox{in }\widetilde{\Omega},\\
A^{ij}(x)  \partial_j u(x) \nu_i &= 0 \quad \mbox{on}~\partial{D}_{i},~i=1,2,\\
 u &= \varphi \quad \mbox{on } \partial \Omega,
\end{aligned}
\right.
\end{equation}
where $0 < \gamma < 1$ and $(A^{ij}(x))$ satisfies, for some constants $\sigma > 0$,
\begin{equation}
\label{A_ij_assumption}
(A^{ij}(x)) \in C^{\gamma}~~\mbox{is symmetric}, \quad \sigma I \le A(x) \le \frac{1}{\sigma} I,
\end{equation}
$\varphi\in{C}^{2}(\partial\Omega)$ is given, and $\nu = (\nu_1, \ldots, \nu_n)$ denotes the inner normal vector on $\partial{D}_{1} \cup \partial{D}_{2}$. We use the notation $x = (x', x_n)$, where $x' \in \bR^{n-1}$. After choosing a coordinate system properly, we can assume that near the origin,  the part of $\partial D_1$ and $\partial D_2$, denoted by $\Gamma_+$ and $\Gamma_-$, are respectively the graphs of two $C^{2,\gamma}$ functions in terms of $x'$. That is, for some $R_0 > 0$,
\begin{align}\label{Gamma_plusminus}
\Gamma_+ = \left\{ x_n = \frac{\varepsilon}{2}+f(x'),~|x'|<R_0\right\}~~ \mbox{and} ~~\Gamma_- = \left\{ x_n = -\frac{\varepsilon}{2}+g(x'),~|x'|<R_0\right\},
\end{align}
where $f$ and $g$ are $C^{2,\gamma}(0 < \gamma < 1)$ functions satisfying
\begin{equation}\label{fg_0}
f(x')>g(x')\quad\mbox{for}~~0<|x'|<R_{0},
\end{equation}
\begin{equation}\label{fg_1}
f(0')=g(0')=0,\quad\nabla_{x'}f(0')=\nabla_{x'}g(0')=0, \quad D^2 (f-g)(0') > 0.
\end{equation}
For $0 < r\leq R_{0}$, we denote
\begin{align}\label{domain_def_Omega}
\Omega_{r}:=\left\{(x',x_{n})\in \widetilde{\Omega}~\big|~-\frac{\varepsilon}{2}+g(x')<x_{n}<\frac{\varepsilon}{2}+f(x'),~|x'|<r\right\}.
\end{align}
We will focus on the following problem near the origin:
\begin{equation}\label{main_problem_narrow}
\left\{
\begin{aligned}
-\partial_i (A^{ij}(x)  \partial_j u(x)) &=0 \quad \mbox{in }\Omega_{R_0},\\
A^{ij}(x)  \partial_j u(x) \nu_i &= 0 \quad \mbox{on } \Gamma_+ \cup \Gamma_-.
\end{aligned}
\right.
\end{equation}
It was proved in \cite{BLY2} that for $u \in H^1(\Omega_{R_0})$ satisfying \eqref{main_problem_narrow},
\begin{equation}\label{grad_u_bound_rough}
| \nabla u(x)| \le C \| u\|_{L^\infty(\Omega_{R_0})}(\varepsilon + |x'|^2)^{-1/2}, \quad \forall x \in \Omega_{R_0/2},
\end{equation}
where $C$ is a positive constant depending only on $n, R_0, \gamma, \sigma, \|A\|_{C^{\gamma}}, \|f\|_{C^{2,\gamma}}$, and $\|g\|_{C^{2,\gamma}}$, and is in particular independent of $\varepsilon$.

In this paper, we show that the optimal exponent of the gradient estimates of the insulated conductivity problem \eqref{main_problem_narrow} is closely related to the following eigenvalue problem on $\bS^{n-2}$. Consider
\begin{equation}
\label{SL_problem}
-\dv_{\bS^{n-2}}\Big(a(\xi)\nabla_{\bS^{n-2}}u (\xi)\Big) =\lambda a(\xi) u(\xi), \quad \xi \in \bS^{n-2},
\end{equation}
where $a(\xi)$ is a positive function on $\bS^{n-2}$ with $\ln a \in L^\infty(\bS^{n-2})$.
Denote
\begin{equation}
\label{inner_product}
\langle u, v\rangle_{\bS^{n-2}}=\fint_{\bS^{n-2}} a(\xi) uv\,d\sigma.
\end{equation}
From the classical theory, all eigenvalues are real, and the corresponding eigenfunctions can be normalized to form an orthonormal basis of $L^2(\bS^{n-2})$ under the inner-product defined above.
The first nonzero eigenvalue $\lambda_1$ of this problem is given by the Rayleigh quotient:
\begin{equation}
\label{first_eigenvalue}
\lambda_1=\inf_{u\not\equiv 0,\langle u, 1\rangle_{\bS^{n-2}}=0}
\frac{\fint_{\bS^{n-2}}a(\xi)|\nabla_{\bS^{n-2}}u|^2\,d\sigma}
{\fint_{\bS^{n-2}}a(\xi)|u|^2\,d\sigma}.
\end{equation}
Let $\alpha(\lambda_1)$ be the positive root of the quadratic polynomial $\alpha^2 + (n-1)\alpha - \lambda_1$, that is,
\begin{equation}
\label{alpha_lambda_1}
\alpha(\lambda_1) = \frac{-(n-1) + \sqrt{(n-1)^2 + 4 \lambda_1}}{2}.
\end{equation}

First, we consider the case when $A^{ij}(0) = \delta_{ij}$, where $\delta_{ij}$ is the Kronecker delta function. When two inclusions touch, namely, $\varepsilon = 0$, we prove the following gradient estimates.
\begin{theorem}
\label{touch_thm}
For $n \ge 3$, $R_0 > 0$, and $\varepsilon = 0$, let $f, g \in C^{2,\gamma} (0 < \gamma < 1)$ satisfy \eqref{fg_0} and \eqref{fg_1}, $(A^{ij}(x))$ satisfy \eqref{A_ij_assumption} with $\sigma > 0$ in $\Omega_{R_0}$, and $A^{ij}(0) = \delta_{ij}$. For any solution $u \in H^1(\Omega_{R_0})$ of \eqref{main_problem_narrow}, we have
\begin{equation}\label{main_goal_touch_case}
|\nabla u(x)|\le C \|u\|_{L^\infty(\Omega_{R_0})} |x'|^{\alpha(\lambda_1)-1}\quad \forall x\in \Omega_{R_0/2} \setminus\{0\},
\end{equation}
where $\lambda_1$ and $\alpha(\lambda_1)$ are given by \eqref{first_eigenvalue} and \eqref{alpha_lambda_1} with $a(\xi)= \xi^t \left( D^2(f-g)(0') \right) \xi$, and $C$ is a positive constant depending only on $n$, $R_0$, $\gamma$, $\sigma$, a positive lower bound of the eigenvalues of $D^2 (f-g)(0')$, and upper bounds of $\|A\|_{C^{\gamma}}$, $\|f\|_{C^{2,\gamma}}$, and $\|g\|_{C^{2,\gamma}}$.
\end{theorem}

\begin{remark}
When $a(\xi) > 0$ a.e. satisfies $\ln a \in L^\infty(\bS^{n-2})$ and $\int_{\bS^{n-2}} a x_i = 0$ for all $i = 1,\ldots, n-1$, it will be shown that $\lambda_1 \le n - 2$ (see Lemma \ref{lambda_1_upper_bound}), and hence $\alpha(\lambda_1) \in (0,1)$.
\end{remark}

When $\varepsilon > 0$, the following gradient estimate is proved.

\begin{theorem}
\label{gen_thm}
For $n \ge 3$, $R_0 > 0$, and $\varepsilon \in (0,1/4)$, let $(A^{ij}(x))$, $f$, $g$, $\lambda_1$, and $\alpha(\lambda_1)$ be the same as in Theorem \ref{touch_thm}. For any solution $u \in H^1(\Omega_{R_0})$ of \eqref{main_problem_narrow}, we have, for any $0 \le \alpha < \alpha(\lambda_1)$,
\begin{equation*}
|\nabla u(x)|\le C \|u\|_{L^\infty(\Omega_{ R_0})} (\varepsilon + |x'|^2)^{\frac{\alpha-1}{2}}\quad \forall x\in \Omega_{R_0/2},
\end{equation*}
where $C$ is a positive constant depending only on $n$, $R_0$, $\gamma$, $\sigma$, a positive lower bound of $\alpha(\lambda_1) - \alpha$, a positive lower bound of the eigenvalues of $D^2 (f-g)(0')$, and upper bounds of $\|A\|_{C^{\gamma}}$, $\|f\|_{C^{2,\gamma}}$, and $\|g\|_{C^{2,\gamma}}$.
\end{theorem}

We show that the estimate \eqref{main_goal_touch_case} is optimal in the following sense. Note that in the next three theorems, $\partial D_1$ and $\partial D_2$ near the origin are represented by the graphs of $f$ and $g$ respectively.

\begin{theorem}
\label{optimality_3d}
For $n = 3$, $A^{ij}(x) \equiv \delta_{ij}$, and for any positive definite matrix $M$, there exist smooth strictly convex inclusions $D_1, D_2$ inside $\Omega = B_5$ with $D^2(f-g)(0') = M$, and a boundary data $\varphi \in C^\infty(\partial \Omega)$ with $\| \varphi \|_{L^\infty(\partial \Omega)} = 1$, such that the solution $u \in H^1(\widetilde \Omega)$ of \eqref{general_equation} satisfies
$$
\limsup_{x \in \widetilde\Omega, |x| \to 0} |x'|^{1- \alpha(\lambda_1)}|\nabla u (x)| > \frac{1}{C},
$$
where $\lambda_1$ and $\alpha(\lambda_1)$ are given by \eqref{first_eigenvalue} and \eqref{alpha_lambda_1} with $a(\xi)= \xi^t M \xi$, and $C$ is some positive constant depending only on the positive lower bounds of the eigenvalues of $M$, and upper bounds of $\|\partial D_1 \|_{C^{4}}$ and $\|\partial D_2\|_{C^{4}}$.
\end{theorem}

\begin{theorem}
\label{optimality_higher_dimension}
For $n \ge 4$ and $A^{ij}(x) \equiv \delta_{ij}$, there exists an $\varepsilon_0 = \varepsilon_0(n) \in (0 ,1/2)$ such that for any positive definite matrix $M$ satisfying
$$
(1-\varepsilon_0) \frac{I}{\|I\|} \le \frac{M}{\|M\|} \le (1+\varepsilon_0) \frac{I}{\|I\|},
$$
there exist smooth strictly convex inclusions $D_1, D_2$ inside $\Omega = B_5$ with $D^2(f-g)(0') = M$, and a boundary data $\varphi \in C^\infty(\partial \Omega)$ with $\| \varphi \|_{L^\infty(\partial \Omega)} = 1$, such that the solution $u \in H^1(\widetilde \Omega)$ of \eqref{general_equation} satisfies
$$
\limsup_{x \in \widetilde\Omega, |x| \to 0}  |x'|^{1- \alpha(\lambda_1)}|\nabla u (x)| > \frac{1}{C},
$$
where $\lambda_1$ and $\alpha(\lambda_1)$ are given by \eqref{first_eigenvalue} and \eqref{alpha_lambda_1} with $a(\xi)= \xi^t M \xi$, and $C$ is a positive constant depending only on $n$, $\|M\|$, and upper bounds of $\|\partial D_1 \|_{C^{4}}$ and $\|\partial D_2\|_{C^{4}}$.
\end{theorem}
In the above, $\|M\|$ and $\|I\|$ denote the standard norm of the matrices. Theorems \ref{optimality_3d} and \ref{optimality_higher_dimension} are consequences of the following more general theorem.

Let $D_1, D_2$ be two strictly convex smooth domains in $B_4 \subset \bR^n$, which are symmetric in $x_j$ for each $1 \le j \le n - 1$, and $\overline D_1 \cap \overline D_2 = \{0\}$. Let $\Omega = B_{5}$ and $\widetilde{\Omega} = \Omega \setminus \{ \overline{ D_1 \cup  D_2} \}$.

\begin{theorem}
\label{optimality_touch_case_thm}
For $n \ge 3$, let $D_1$, $D_2$, and $\Omega$ be as above, $A^{ij}(x) \equiv \delta_{ij}$, $\lambda_1$ and $\alpha(\lambda_1)$ be given by \eqref{first_eigenvalue} and \eqref{alpha_lambda_1} with $a(\xi)= \xi^t \left( D^2(f-g)(0') \right) \xi$. Assume that the eigenspace corresponding to $\lambda_1$ contains a function which is odd in $x_j$ for some $1 \le j \le n - 1$. Let $\varphi = x_j$ and $u \in H^1(\widetilde \Omega)$ be the solution of \eqref{general_equation}. Then
$$
\limsup_{x \in \widetilde\Omega, |x| \to 0}  |x'|^{1- \alpha(\lambda_1)}|\nabla u (x)| > \frac{1}{C},
$$
where $C$ is some positive constant depending only on $n$, a positive lower bound of the eigenvalues of $D^2 (f-g)(0')$, and upper bounds of $\|\partial D_1 \|_{C^{4}}$ and $\|\partial D_2\|_{C^{4}}$.
\end{theorem}

We will show in Section 5 (see Theorem \ref{3d_eigenvalue_thm} and Corollary \ref{property_o_corollary}) that the conditions in Theorems \ref{optimality_3d} and \ref{optimality_higher_dimension} imply the condition in Theorem \ref{optimality_touch_case_thm}.

The rest of this paper is organized as follows. In Section 2, we establish some estimates for the associated degenerate elliptic operator
$$L_\varepsilon := \dv\Big[ \Big( \varepsilon + a\Big(\frac{x'}{|x'|}\Big)|x'|^2 \Big) \nabla \Big],$$ which play an important role in proving Theorems \ref{touch_thm} and \ref{gen_thm}. Theorems \ref{touch_thm} and \ref{gen_thm} are proved in Sections 3 and 4, respectively. Some properties of $\lambda_1$, the first nonzero eigenvalue of \eqref{SL_problem}, and its corresponding eigenspace, are established in Section 5. Theorem \ref{optimality_touch_case_thm} is proved in Section 6, and therefore, Theorems \ref{optimality_3d} and \ref{optimality_higher_dimension} follow. Finally in Section 7, we discuss the case when $A_{ij}(0) \neq \delta_{ij}$, and give a reduction to the case when $A_{ij}(0) = \delta_{ij}$.

\section{Some estimates on the associated degenerate elliptic operator}

In this section, we establish some estimates that are useful in proving Theorems \ref{touch_thm} and \ref{gen_thm}. Throughout the section, we work in the domain $B_{R_0} \subset \bR^{n-1}$ for some $R_0 > 0$ and $n \ge 3$. Let $a$ be a function on $\bS^{n-2}$ satisfying
\begin{equation}
\label{a_assumption}
a > 0~~\mbox{a.e.} \quad\text{and}\quad  \ln a \in L^\infty(\bS^{n-2}),
\end{equation}
and let $\lambda_1$ and $\alpha(\lambda_1)$ be given by \eqref{first_eigenvalue} and \eqref{alpha_lambda_1}.

Here are some notation we will be using throughout this paper:
For $\sigma, s \in \bR$, we introduce the following norm
\begin{equation}
                \label{eq2.24}
\| F \|_{\varepsilon, \sigma,s,B_{R_0}}: = \sup_{x' \in B_{R_0}} \frac{|F(x')|}{|x'|^{\sigma} (\varepsilon + |x'|^2)^{1-s}}.
\end{equation}
For any bounded set $\Omega \subset \bR^{n-1}$, we denote $H^1(\Omega, |x'|^2 dx')$ to be the following weighted $H^1$ norm:
$$
\|f\|_{H^1(\Omega, |x'|^2 dx')} = \left( \int_{\Omega} |f|^2|x'|^2 \, dx' \right)^{\frac{1}{2}} + \left( \int_{\Omega} |\nabla f|^2|x'|^2 \, dx' \right)^{\frac{1}{2}}.
$$
For any $0 < \rho < R_0$, we denote
\begin{align*}
(u)_{\partial B_\rho}^a &:= \left(  \int_{\partial B_\rho} a\Big(\frac{x'}{|x'|}\Big) \, d\sigma \right)^{-1} \int_{\partial B_\rho} a\Big(\frac{x'}{|x'|}\Big) u(x') \, d\sigma,\\
(u)_{B_\rho}^a &:= \left(  \int_{ B_\rho} a\Big(\frac{x'}{|x'|}\Big) \, dx' \right)^{-1} \int_{B_\rho} a\Big(\frac{x'}{|x'|}\Big) u(x') \, dx'.
\end{align*}

\begin{proposition}
\label{prop_grad_u_bar_control}
For $n \ge 3$, let $a$ satisfy \eqref{a_assumption}, $\lambda_1$ and $\alpha(\lambda_1)$ be given by \eqref{first_eigenvalue} and \eqref{alpha_lambda_1}. For $\sigma > 1$, $\sigma-1 \neq \alpha(\lambda_1)$,  let $\bar{u} \in H^1(B_{R_0}, |x'|^2 dx')$ be a solution of
\begin{equation*}
\dv\Big[a\Big(\frac{x'}{|x'|}\Big)|x'|^2\nabla \bar u\Big]= \dv F + G \quad\text{in}\,\,B_{R_0}\subset \bR^{n-1},
\end{equation*}
where $F,G \in L^\infty(B_{R_0})$ satisfy
\begin{equation}
\label{FG_bound}
\|F\|_{\varepsilon, \sigma,1,B_{R_0}}< \infty , \quad \|G\|_{\varepsilon, \sigma-1,1,B_{R_0}} < \infty.
\end{equation}
Then $\bar{u} \in C^\beta(B_{R_0})$ for some $\beta\in (0,1)$. Moreover, for any $|x'| \le R_0/2$, we have
\begin{align}
                        \label{eq4.58}
&|\bar u(x') - \bar u(0)|\nonumber \\
 &\le C(\|F\|_{\varepsilon, \sigma,1,B_{R_0}} + \|G\|_{\varepsilon, \sigma-1,1,B_{R_0}} +  \| \bar u - \bar u(0) \|_{L^2(\partial B_{R_0})})  |x'|^{\tilde\alpha}
\end{align}
where $\tilde\alpha := \min \{ \alpha(\lambda_1) , \sigma - 1\}$, and $C$ is some positive constant depending only on $n$, $\sigma$, $R_0$, an upper bound of $\|\ln a\|_{L^\infty}$, and is independent of $\varepsilon$.
\end{proposition}

For the proof, we use an iteration argument based on the following two lemmas.

\begin{lemma}
\label{lemma_v1_gen}
For $n \ge 3$, let $a$ satisfy \eqref{a_assumption}, $\lambda_1$ and $\alpha(\lambda_1)$ be given by \eqref{first_eigenvalue} and \eqref{alpha_lambda_1}, and $v_1 \in H^1(B_{R_0}, |x'|^2 dx')$ satisfy
\begin{equation}
\label{equation_v1_gen}
\dv\Big[a\Big(\frac{x'}{|x'|}\Big)|x'|^2\nabla v_1\Big]=0\quad\text{in}\,\,B_{R_0} \subset \bR^{n-1}.
\end{equation}
Then $v_1 \in C^\beta(B_{R_0})$, for some $\beta > 0$ depending only on $n$ and $\| \ln a\|_{L^\infty}$. Moreover, for any $0 < \rho < R \le R_0$, we have
$$
v_1(0) = (v_1)_{\partial B_\rho}^a,
$$
\begin{align}\label{v1_L2_decay}
&\left( \fint_{\partial B_\rho} a\left( \frac{x'}{|x'|} \right)|v_1(x') - v_1(0)|^2 \, d\sigma \right)^{\frac{1}{2}}\nonumber\\
&\le \left(\frac{\rho}{R} \right)^{\alpha(\lambda_1)} \left( \fint_{\partial B_R} a\left( \frac{x'}{|x'|} \right) |v_1(x') - v_1(0)|^2 \, d\sigma \right)^{\frac{1}{2}},
\end{align}
and for any $x' \in B_{R_0/2}$,
\begin{equation}
\label{v1_holder}
|v_1(x') - v_1(0)| \le C R_0^{-\alpha(\lambda_1) - \frac{n-1}{2}} \| v_1 - v_1(0) \|_{L^2(\partial B_{R_0})} |x'|^{\alpha(\lambda_1)},
\end{equation}
where $C$ is some positive constant depending only on $n$ and $\| \ln a\|_{L^\infty}$.
\end{lemma}

\begin{proof}
By \cite{FKS}*{Theorem 2.3.12}, $v_1 \in C^\beta(B_{R_0})$ for some $\beta > 0$.
It should be noted that when $n=3$, the weight $|x'|^2$ does not satisfy the $A_2$ condition (in $\bR^{n-1}$) required in \cite{FKS}*{Theorem 2.3.12}. Nevertheless, it satisfies the conditions in \cite{FKS}*{Section 3, pp. 106}. Therefore, the H\"older estimate still holds.
Without loss of generality, it suffices to prove \eqref{v1_L2_decay} and \eqref{v1_holder} for $a \in C^\infty(\bS^{n-2})$ and $R=R_0= 1$. In the polar coordinates, we write $x' = (r ,\xi)$ with $0 < r < 1, \xi \in \bS^{n-2}$. Let $\varphi(r) \in C^\infty_0((0,1))$ and $\psi(\xi) \in C^\infty(\bS^{n-2})$. Multiplying \eqref{equation_v1_gen} by $\varphi \psi$ and integrating by parts gives
\begin{align*}
\int_{B_1} a(\xi) r^2 \nabla v_1 \cdot \nabla (\varphi \psi) &= \int_0^1 \int_{\bS^{n-2}} a r^n \partial_r v_1 \varphi' \psi + ar^{n-2} \nabla_{\bS^{n-2}} v_1 \cdot \nabla_{\bS^{n-2}} \psi \varphi \, d\xi dr\\
&= -\int_{B_1} [ar^2 \partial_{rr} v_1 + nar \partial_r v + \dv_{\bS^{n-2}} (a \nabla_{\bS^{n-2}} v_1)] \varphi \psi.
\end{align*}
Therefore, we can write \eqref{equation_v1_gen} in polar coordinates as
\begin{equation}
\label{equation_v1_gen_polar_coordinate}
\partial_{rr} v_1 + \frac {n}r \partial_{r} v_1 +\frac 1 {a(\xi)r^2}\dv_{\bS^{n-2}}\Big(a(\xi)\nabla_{\bS^{n-2}}v_1 \Big)=0 \quad \mbox{in}~~B_1 \setminus \{0\},
\end{equation}
Let $\lambda_0 = 0$, $\{\lambda_k\}_{k = 1}^\infty$ be the set of all positive eigenvalues of \eqref{SL_problem} satisfying $\lambda_k < \lambda_{k+1}$ for all $k \in \bN \cup \{0\}$. Let $Y_0$ be the positive constant satisfying $\langle Y_0, Y_0 \rangle_{\bS^{n-2}} = 1$, $Y_{k,i}$ be an eigenfunction corresponding to $\lambda_k$, that is,
$$
\dv_{\bS^{n-2}}\Big(a(\xi)\nabla_{\bS^{n-2}}Y_{k,i}\Big) =-\lambda_k a(\xi) Y_{k,i},
$$
and $\{Y_{k,i}\}_{k,i} \cup \{Y_0\}$ forms an orthonormal basis of $L^2(\bS^{n-2})$ with respect to the inner product \eqref{inner_product}.

For $0 < r < 1$, take the decomposition
\begin{equation}
\label{v1_gen_expansion}
v_1(r, \xi) = V_0(r)Y_0 + \sum_{k=1}^\infty \sum_{i=1}^{N(k)} V_{k,i}(r)Y_{k,i}(\xi)\quad \mbox{in}~~L^2(\bS^{n-2}),
\end{equation}
where $V_{0}(r), V_{k,i}(r) \in C^2(0,1)$ are given by
$$
V_{0}(r) = \fint_{\bS^{n-2}} a(\xi) v_1(r, \xi) Y_{0} \, d\xi, \quad V_{k,i}(r) = \fint_{\bS^{n-2}} a(\xi) v_1(r, \xi) Y_{k,i}(\xi) \, d\xi.
$$
Multiplying \eqref{equation_v1_gen_polar_coordinate} by $a(\xi)Y_{0}$ and $a(\xi)Y_{k,i}(\xi)$ respectively and integrate over $\bS^{n-2}$, we see that $V_{0}(r)$ and $V_{k,i}(r)$ satisfy
$$
V_{0}'' + \frac {n}r V_{0}' =0 \quad \mbox{and}\quad V_{k,i}'' + \frac {n}r V_{k,i}' -\frac{\lambda_k}{r^2}V_{k,i} =0 \quad \mbox{in}~~(0,1).
$$
Therefore $V_{0} = c_1 + c_2r^{1-n}$ and $V_{k,i} = c_3 r^{\alpha(\lambda_k)_+} + c_4r^{\alpha(\lambda_k)_-}$ for some constants $c_1, c_2, c_3,$ and $c_4$, where
$$
\alpha(\lambda_k)_\pm := \frac{-(n-1) \pm \sqrt{(n-1)^2 + 4 \lambda_k}}{2}.
$$
Since $v_1 \in H^1(B_{1}, r^2 dx')$, we have for any $0 < \delta < 1$,
\begin{align*}
\infty > \int_{B_1 \setminus B_\delta} a(\xi) v_1^2 r^2 \, dx' &\ge \frac{1}{C} \int_{B_1 \setminus B_\delta} a(\xi) V_0(r)^2 r^2 \, dx'\\
&\ge \frac{1}{C} \int_\delta^1 |c_1 + c_2r^{1-n}|^2 r^{n+1} \, dr,
\end{align*}
which implies $c_2 \equiv 0$. Hence $V_0(r) \equiv V_0(1)$.
Similarly, we have $c_4 \equiv 0$, and hence
$$
V_{k,i}(r) = r^{\alpha(\lambda_k)_+}V_{k,i}(1).
$$
By \eqref{v1_gen_expansion}, for any $0 < \rho < 1$,
\begin{align*}
\fint_{\partial B_\rho} a\left( \frac{x'}{|x'|} \right)|v_1(x') - V_0 Y_0|^2 \, d\sigma  &= \sum_{k=1}^\infty \sum_{i=1}^{N(k)}  |V_{k,i}(\rho)|^2 \\
&\le \rho^{2\alpha(\lambda_1)} \sum_{k=1}^\infty \sum_{i=1}^{N(k)}|V_{k,i}(1)|^2 \\
&= \rho^{2\alpha(\lambda_1)} \fint_{\partial B_1} a\left( \frac{x'}{|x'|} \right)|v_1(x') - V_0 Y_0|^2 \, d\sigma.
\end{align*}
Therefore, $v_1(0) = V_0(0)Y_0 = V_0(\rho) Y_0 = (v_1)_{\partial B_\rho}^a$ for any $\rho \in (0,1)$, and
$$
\left( \fint_{\partial B_\rho} |v_1(x') - v_1(0)|^2 \, d\sigma \right)^{1/2} \le C \rho^{\alpha(\lambda_1)}\left( \fint_{\partial B_1} |v_1(x') - v_1(0)|^2 \, d\sigma \right)^{1/2},
$$
which implies \eqref{v1_holder} by the interior elliptic estimate applied to $B_\rho \setminus \overline{B}_{\rho/2}$.
\end{proof}

\begin{lemma}
\label{lemma_v2}
For $n \ge 3$ and $\sigma > 1$. Let $a$ satisfy \eqref{a_assumption}, $\lambda_1$ and $\alpha(\lambda_1)$ be given by \eqref{first_eigenvalue} and \eqref{alpha_lambda_1}, and $v_2 \in H^1_0(B_{R_0}, |x'|^2 dx')$ satisfy
\begin{equation}
\label{equation_v2}
\dv\Big[a\Big(\frac{x'}{|x'|}\Big)|x'|^2\nabla v_2\Big]=\dv F + G\quad\text{in}\,\,B_{R_0}\subset \bR^{n-1},
\end{equation}
where $F,G \in L^\infty(B_{R_0})$ satisfy \eqref{FG_bound}.
Then we have
$$
\|v_2\|_{L_\infty(B_{R_0})}\le C(\|F\|_{\varepsilon,\sigma,1, B_{R_0}} + \|G\|_{\varepsilon,\sigma-1,1,B_{R_0}}),
$$
where $C>0$ depends only on $n$, $\sigma$, and an upper bound of $\|\ln a\|_{L^\infty}$ and is in particular independent of $\varepsilon$.
\end{lemma}
\begin{proof}
Without loss of generality, we assume $R_0 = 1$ and
$$\|F\|_{\varepsilon,\sigma,1,B_{1}} + \|G\|_{\varepsilon,\sigma-1,1,B_{1}} = 1.$$ For $p \ge 2$, we multiply the equation \eqref{equation_v2} with $-|v_2|^{p-2}v_2$ and integrate by parts to obtain
\begin{align*}
&(p-1) \int_{B_{1}} a\Big(\frac{x'}{|x'|}\Big)|x'|^2 |\nabla v_2|^{2}|v_2|^{p-2}\,dx'\\
&= (p-1)\int_{B_{1}}F \cdot \nabla v_2|v_2|^{p-2}\,dx' - \int_{B_{1}} G v_2 |v_2|^{p-2} \, dx'.
\end{align*}
By the definition in \eqref{eq2.24},
$$
|F(x')| \le |x'|^{\sigma}\|F\|_{\varepsilon,\sigma,1, B_{1}} \quad \mbox{and} \quad |G(x')| \le |x'|^{\sigma-1} \|G\|_{\varepsilon,\sigma-1,1,B_{1}} \quad \mbox{for}~~x' \in B_1.
$$
Therefore, by Young's inequality, H\"older's inequality, and using $\sigma > 1$,
\begin{align*}
&(p-1)\left|\int_{B_{1}}F \cdot \nabla v_2|v_2|^{p-2}\,dx' \right|\\
&\le \frac{(p-1)\delta}{2}\int_{B_{1}}|x'|^2 |\nabla v_2|^{2}|v_2|^{p-2}\,dx' +  C(p-1)\int_{B_{1}}|x'|^{2\sigma-2} |v_2|^{p-2}\,dx'\\
&\le \frac{(p-1)\delta}{2}\int_{B_{1}}|x'|^2 |\nabla v_2|^{2}|v_2|^{p-2}\,dx' + \\
&\quad + C(p-1) \left( \int_{B_{1}}|x'|^{(\sigma-1)(n+1+2\mu) - (n-1 + 2\mu)} \, dx' \right)^{\frac{2}{n+1+2\mu}} \times \\
& \quad \times \left( \int_{B_{1}} |x'|^2 |v_2|^{(p-2) \frac{n+1+2\mu}{n-1+2\mu}}  \, dx'\right)^{\frac{n-1+2\mu}{n+1+2\mu}},
\end{align*}
and
\begin{align*}
\left| \int_{B_{1}} G v_2 |v_2|^{p-2} \, dx' \right| &\le C\left( \int_{B_{1}}|x'|^{(\sigma-1)(n+1+2\mu)/2 - (n-1 + 2\mu)} \, dx' \right)^{\frac{2}{n+1+2\mu}} \times \\
& \quad \times \left( \int_{B_{1}} |x'|^2 |v_2|^{(p-1) \frac{n+1+2\mu}{n-1+2\mu}}  \, dx'\right)^{\frac{n-1+2\mu}{n+1+2\mu}},
\end{align*}
where $\mu > 0$ is chosen sufficiently small so that
$$\int_{B_{1}}|x'|^{(\sigma-1)(n+1+2\mu)/2 - (n-1 + 2\mu)} \, dx' < \infty.$$
Hence,
\begin{align}
\label{moser_inequality1}
&\frac{4(p-1)}{p^2} \int_{B_{1}} |x'|^2 \left| \nabla |v_2|^{\frac{p}{2}} \right|^2\,dx' = (p-1)\int_{B_{1}} |x'|^2 |\nabla v_2|^{2}|v_2|^{p-2}\,dx'\nonumber\\
&\le C(p-1)  \|v_2^{p-2}\|_{L^{\frac{n+1+2\mu}{n-1+2\mu}}(B_{1}, |x'|^2dx')} + C \|v_2^{p-1}\|_{L^{\frac{n+1+2\mu}{n-1+2\mu}}(B_{1}, |x'|^2dx')}.
\end{align}
We use the following version of the Caffarelli-Kohn-Nirenberg inequality (see \cite{CKN}):
\begin{equation}
\label{CKN_inequality}
 \|u\|_{L^{\frac{2(n+1)}{n-1}}(B_{1}, |x'|^2dx')} \le C  \|\nabla u\|_{L^{2}(B_{1}, |x'|^2dx')} \quad \forall u \in H_0^1(B_1, |x'|^2dx').
\end{equation}
Taking $p = 2$ in \eqref{moser_inequality1}, we have, by \eqref{CKN_inequality} with $u = |v_2|$ and H\"older's inequality,
\begin{equation}
\label{starting_point}
\|v_2\|_{L^{\frac{2(n+1+2\mu)}{n-1+2\mu}}(B_{1}, |x'|^2dx')} \le C.
\end{equation}
For $p \ge 2$, from \eqref{moser_inequality1}, by \eqref{CKN_inequality} with $u = |v_2|^{\frac{p}{2}}$ and H\"older's inequality,
\begin{align*}
&\|v_2\|_{L^{\frac{(n+1)p}{n-1}}(B_{1}, |x'|^2dx')}^p \le C\| \nabla |v_2|^{\frac{p}{2}} \|_{L^{2}(B_{1}, |x'|^2dx')}^{2}\\
&\le C p^2 \|v_2\|_{L^{\frac{n+1+2\mu}{n-1+2\mu}(p-2)}(B_{1}, |x'|^2dx')}^{p-2} + Cp\|v_2\|_{L^{\frac{n+1+2\mu}{n-1+2\mu}(p-1)}(B_{1}, |x'|^2dx')}^{p-1}\\
&\le \max_{i=1,2} Cp^i \|v_2\|_{L^{\frac{n+1+2\mu}{n-1+2\mu}p}(B_{1}, |x'|^2dx')}^{p-i}.
\end{align*}
By Young's inequality,
\begin{align*}
\|v_2\|_{L^{\frac{(n+1)p}{n-1}}(B_{1}, |x'|^2dx')} &\le \max_{i=1,2} (Cp^i)^{1/p} \left( \frac{p-i}{p} \|v_2\|_{L^{\frac{n+1+2\mu}{n-1+2\mu}p}(B_{1}, |x'|^2dx')} + \frac{i}{p} \right)\\
&\le (Cp^2)^{1/p} \left( \|v_2\|_{L^{\frac{n+1+2\mu}{n-1+2\mu}p}(B_{1}, |x'|^2dx')} + \frac{2}{p} \right).
\end{align*}
For $k \ge 0$, let
$$
p_k = 2 \left( \frac{n+1}{n-1} \cdot \frac{n-1+2\mu}{n+1+2\mu} \right)^k \frac{n+1+2\mu}{n-1+2\mu}.
$$
Iterating the relations above, we have, by \eqref{starting_point},
\begin{align}
\label{Moser_iteration}
\|v_2\|_{L^{p_k}(B_{1}, |x'|^2dx')} &\le \prod_{i=0}^{k-1} \left( Cp_i^2 \right)^{2/p_i}\|v_2\|_{L^{p_0}(B_{1}, |x'|^2dx')}\nonumber\\
&+ \sum_{i=0}^{k-1} \prod_{j=i}^{k-1} \left( Cp_{j}^2 \right)^{2/p_{j}}\frac{4}{p_i} \nonumber\\
&\le C \|v_2\|_{L^{\frac{2(n-1+2\mu)}{n-3+2\mu}}(B_{1}, |x'|^2dx')} + C \sum_{i=0}^{k-1} \frac{1}{p_i} \le C,
\end{align}
where $C$ is a positive constant depending on $n$ and $\sigma$, and is in particular independent of $k$. The lemma is concluded by taking $k \to \infty$ in \eqref{Moser_iteration}.
\end{proof}

Now we are in a position to prove Proposition \ref{prop_grad_u_bar_control}.

\begin{proof}[Proof of Proposition \ref{prop_grad_u_bar_control}]
We first show the (H\"older) continuity of $\bar u$. By Lemma \ref{lemma_v1_gen}, $v_1$ is locally H\"older continuous. In particular, it satisfies the estimate in \cite{FKS}*{Lemma 2.3.11}.
Now for $F$ and $G$ such that
$$
Fr^{-2}\in L^{p}(B_{R_0},r^2\,dx')\quad \text{and}\quad
Gr^{-2}\in L^{p/2}(B_{R_0},r^2\,dx')
$$ with some $p>n+2$, which are satisfied by the condition of the proposition,
we can apply the Moser iteration to get an $L^\infty$ estimate of $v_2$ as in \cite{FKS}*{Lemma 2.3.14}.
By combining these two estimates and using the standard iteration argument, we get the local H\"older continuity of $\bar u$ with a small exponent. The proof of \eqref{eq4.58} below essentially follows this scheme, using the more precise $L^2$ oscillation estimate obtained in Lemma \ref{lemma_v1_gen}.

Without loss of generality, we assume that $\bar{u}(0) = 0$ and $$
\|F\|_{\varepsilon, \sigma,1,B_{R_0}}+\|G\|_{\varepsilon, \sigma-1,1,B_{R_0}} +  \| \bar u  \|_{L^2(\partial B_{R_0})}=1.
$$
Consider
$$
\omega(\rho):=\Big(\fint_{\partial B_\rho}a\Big(\frac{x'}{|x'|}\Big)|\bar u(x')|^2\,d\sigma\Big)^{1/2}.
$$
For $0 < \rho \le R/2 \le R_0/2$, we write $\bar u=v_1+v_2$ in $B_R$, where $v_2$ satisfies
$$
\dv\Big[a\Big(\frac{x'}{|x'|}\Big)|x'|^2\nabla v_2\Big]=\dv F + G\quad\text{in}\,\,B_R
$$
and $v_2=0$ on $\partial B_R$. Thus $v_1$ satisfies
$$
\dv\Big[a\Big(\frac{x'}{|x'|}\Big)|x'|^2\nabla v_1\Big]=0\quad\text{in}\,\,B_R
$$
and $v_1=\bar u$ on $\partial B_R$. By Lemma \ref{lemma_v1_gen},
\begin{align}\label{v1_control}
&\left( \fint_{\partial B_\rho} a\left( \frac{x'}{|x'|} \right)|v_1(x') - v_1(0)|^2 \, d\sigma \right)^{\frac{1}{2}}\nonumber\\
& \le \left(\frac{\rho}{R} \right)^{\alpha(\lambda_1)} \left( \fint_{\partial B_R} a\left( \frac{x'}{|x'|} \right) |v_1(x') - v_1(0)|^2 \, d\sigma \right)^{\frac{1}{2}}.
\end{align}
Since $\tilde v_2(x'):=v_2(Rx')$ satisfies
$$
\dv\Big[a\Big(\frac{x'}{|x'|}\Big)|x'|^2\nabla \tilde v_2\Big]
=\dv \tilde F + \tilde G\quad\text{in}\,\,B_1,
$$
where $\tilde F(x'):=R^{-1}F(Rx')$ and $\tilde{G}(x') := G(Rx')$ satisfy
\begin{align*}
\|\tilde{F}\|_{R^{-2}\varepsilon,\sigma,1,B_1} &= R^{\sigma - 1} \|F\|_{\varepsilon, \sigma,1,B_{R}}, \\
\|\tilde G\|_{R^{-2}\varepsilon, \sigma-1,1,B_{1}} &= R^{\sigma - 1} \|G\|_{\varepsilon, \sigma-1,1,B_{R}},
\end{align*}
we apply Lemma \ref{lemma_v2} with $R_0 = 1$ to $\tilde{v}_2$ to obtain
\begin{equation}
\label{v2_control}
\|v_2\|_{L_\infty(B_R)}\le CR^{\sigma-1}.
\end{equation}
Since $\bar{u}(0) = v_1(0) + v_2(0) = 0$, we have $|v_1(0)|=|v_2(0)|$. Combining \eqref{v1_control} and \eqref{v2_control} yields, using $\bar{u} = v_1 + v_2$, and $\bar{u} = v_1$ on $\partial B_R$,
\begin{align}
\label{omega_iteration}
&\omega(\rho)\nonumber\\
&\le \left( \fint_{\partial B_\rho} a\left( \frac{x'}{|x'|} \right)|v_1(x') - v_1(0)|^2 \, d\sigma \right)^{\frac{1}{2}} + \left( \fint_{\partial B_\rho} a\left( \frac{x'}{|x'|} \right) |v_2(x') - v_2(0)|^2 \, d\sigma \right)^{\frac{1}{2}}\nonumber\\
&\le \left(\frac{\rho}{R} \right)^{\alpha(\lambda_1)} \left( \fint_{\partial B_R}  a\left( \frac{x'}{|x'|} \right) |v_1(x')|^2 \, d\sigma \right)^{\frac{1}{2}} + C |v_1(0)| + C\|v_2\|_{L_\infty(B_R)}\nonumber\\
&\le \left(\frac{\rho}{R} \right)^{\alpha(\lambda_1)} \omega(R) + CR^{\sigma - 1}.
\end{align}
For a positive integer $k$, we take $\rho = 2^{-i-1}R_0$ and $R = 2^{-i}R_0$ in \eqref{omega_iteration} and iterate from $i = 0$ to $k-1$. We have, using $\sigma - 1 \neq \alpha(\lambda_1)$,
\begin{align*}
\omega(2^{-k}R_0) &\le 2^{-k\alpha(\lambda_1)} \omega(R_0) + C\sum_{i=1}^k 2^{-(k-i)\alpha(\lambda_1)} (2^{1-i}R_0)^{\sigma - 1}\\
&\le 2^{-k\alpha(\lambda_1)} \omega(R_0) + C2^{-k\alpha(\lambda_1)} R_0^{\sigma - 1} \frac{1 - 2^{k(\alpha(\lambda_1) -\sigma+1)}}{1 - 2^{\alpha(\lambda_1) - \sigma + 1}}.
\end{align*}
It follows that
$$
\omega(2^{-k}R_0) \le 2^{-k\tilde\alpha} \left( \omega(R_0) + CR_0^{\sigma - 1}\right),
$$
where $\tilde{\alpha} = \min\{ \sigma-1, \alpha(\lambda_1) \}$.
For any $\rho \in (0, R_0/2)$, let $k$ be the positive integer such that $2^{-k-1}R_0 < \rho \le 2^{-k}R_0$. Then
\begin{equation*}
\omega(\rho) \le C \rho^{\tilde\alpha}, \quad \forall \rho \in (0, R_0/2),
\end{equation*}
and hence
$$
\Big(\fint_{\partial B_\rho} |\bar u(x')|^2\,d\sigma\Big)^{1/2} \le C \rho^{\tilde\alpha}, \quad \forall \rho \in (0, R_0/2).
$$
The proposition then follows from the standard interior elliptic estimate applied to $B_\rho \setminus \overline{B}_{\rho/2}$.
\end{proof}

In the remaining part of this section, we consider the case when $\varepsilon>0$.
\begin{proposition}
\label{v_bar_iteration_prop}
For $n \ge 3$, $s > -(n-1)/2$, $0 < t \le 1$, and $\varepsilon > 0$. Let $a$ satisfy \eqref{a_assumption} and be H\"older continuous, $\lambda_1$ and $\alpha(\lambda_1)$ be given by \eqref{first_eigenvalue} and \eqref{alpha_lambda_1}, and $\bar{v} \in H^1(B_{R_0})$ be a solution of
\begin{equation*}
\dv\Big[ \Big( \varepsilon + a\Big(\frac{x'}{|x'|}\Big)|x'|^2 \Big) \nabla \bar v\Big]= \dv F  \quad\text{in}\,\,B_{R_0}\subset \bR^{n-1}
\end{equation*}
satisfying
$$
\|\nabla \bar{v}\|_{\varepsilon,-t,1, B_{R_0}} < \infty,
$$
where $F \in L^\infty(B_{R_0})$ satisfy
$$
\| F \|_{\varepsilon,s,0,B_{R_0}} < \infty.
$$
Then for any $0 < \rho < \frac{1}{4}R \le \frac{1}{4} R_0$, we have
\begin{align}
\label{iteration_gen}
&\Big(\fint_{B_\rho \setminus B_{\rho/2}}|\bar v(x') - (\bar v)^a_{B_\rho \setminus B_{\rho/2}}|^2\,d\sigma\Big)^{1/2}\nonumber\\
&\le C\left( \frac{\rho}{R} \right)^{\alpha(\lambda_1)} \Big(\fint_{B_R \setminus B_{R/2}} |\bar v(x') - (\bar v)^a_{B_R \setminus B_{R/2}}|^2\,d\sigma\Big)^{1/2} \nonumber\\
&\quad + C \left( \frac{R}{\rho} \right)^{\frac{n}{2}}\left(R^{1+s} \left( \frac{\sqrt \varepsilon}{R} + 1 \right) \| F \|_{\varepsilon,s,0,B_{R_0}} + \left( \frac{\varepsilon}{R^2} \right)^{\tilde\alpha(\lambda_1)} R^{1-t}\|\nabla \bar{v}\|_{\varepsilon,-t,1, B_{R_0}} \right),
\end{align}
where
\begin{equation}
\label{tilde_alpha}
\tilde\alpha(\lambda_1) = \left\{
\begin{aligned}
&\left( \alpha(\lambda_1) + \frac{n-1}{2} \right)\Big/2, &&\mbox{when}~~n < 5-2\alpha(\lambda_1);\\
&\mbox{any}~~\alpha < 1,&&\mbox{when}~~n = 5-2\alpha(\lambda_1);\\
&1, &&\mbox{when}~~n > 5-2\alpha(\lambda_1),\\
\end{aligned}
\right.
\end{equation}
and $C$ is some positive constant depending only on $n$, $s$, $t$, $R_0$, and an upper bound of $\|\ln a\|_{L^\infty}$, and is independent of $\varepsilon$.
\end{proposition}

Proposition \ref{v_bar_iteration_prop} will follow from Lemma \ref{lemma_v1_gen} and the following lemma.

\begin{lemma}
\label{calculus_lemma}
For $n \ge 3$, $B_1 \subset \bR^{n-1}$, and $\beta < 1$, let $v \in H^1_0(B_1, |x'|^{2+\beta} dx')$. There exists a positive constant $C$ depending only on $n$ and $\beta$, such that
$$
\sup_{0<r<1} r^n \fint_{\partial B_r} |v|^2 \, d\sigma \le C \int_{B_1} |x'|^{2+ \beta} |\nabla v|^2 \, dx'.
$$
\end{lemma}
\begin{proof}
Without loss of generality, we may assume $v \in C_0^1(B_1)$. Then we have
\begin{align*}
r^2 \int_{\partial B_r} |v|^2 \, d\sigma &\le C\int_{\bS^{n-2}} r^n|v(r,\xi)|^2 \, d\xi = C\int_{\bS^{n-2}} r^n\left( \int_r^1 \partial_s v(s, \xi) \, ds \right)^2 \, d\xi\\
&\le C\int_{\bS^{n-2}} r^n\left( \int_r^1 |\partial_s v(s, \xi)|^2s^\beta \, ds \right) \left( \int_r^1 s^{-\beta} \, ds \right) \, d\xi\\
&\le C \int_0^1  \int_{\bS^{n-2}} s^{n+ \beta} |\partial_s v(s ,\xi)|^2 \, ds d\xi \le C \int_{B_1} |x'|^{2+\beta} |\nabla v|^2 \, dx',
\end{align*}
where in the last two lines, we used H\"older's inequality and the Fubini theorem.
\end{proof}

Now we are in a position to prove Proposition \ref{v_bar_iteration_prop}.

\begin{proof}[Proof of Proposition \ref{v_bar_iteration_prop}]
In this proof, we denote $\alpha = \alpha(\lambda_1)$ and $\tilde{\alpha} = \tilde{\alpha}(\lambda_1)$ for simplicity. Without loss of generality, we may assume $\bar{v}(0) = 0$. Then by the mean value formula,
\begin{equation}
\label{v_bar_pointwise_control}
|\bar{v}(x')| \le |x'|^{1-t}\|\nabla \bar{v}\|_{\varepsilon,-t,1, B_{R_0}} \quad \mbox{for}~~x' \in B_{R_0}.
\end{equation}
For any $0 < R < R_0$, we write $\bar{v} = v_1 + v_2$ so that $v_1 \in H^1(B_R , |x'|^2 dx')$ satisfies
$$
\dv\Big[a\Big(\frac{x'}{|x'|}\Big)|x'|^2\nabla v_1\Big]=0\quad\text{in}\,\,B_{R},
$$
with $v_1 = \bar{v}$ on $\partial B_R$.
By Lemma \ref{lemma_v1_gen} and \eqref{v_bar_pointwise_control}, we have, for $0 < \rho < R$,
$$
\left( \fint_{\partial B_\rho} |v_1(x') - v_1(0)|^2 \, d\sigma \right)^{\frac{1}{2}} \le C \left(\frac{\rho}{R} \right)^{\alpha} R^{1-t} \|\nabla \bar{v}\|_{\varepsilon,-t,1, B_{R_0}}.
$$
Since $a$ is H\"older continuous, by the interior gradient estimate applied in $B_\rho \setminus \overline{B}_{\rho/2}$ for $0 < \rho < R$,
\begin{equation}\label{grad_v1_interior}
|\nabla v_1(x')| \le C\|\nabla \bar{v}\|_{\varepsilon,-t,1, B_{R_0}} |x'|^{\alpha - 1} R^{1-t - \alpha} \quad \mbox{for}~~x' \in B_{R/2} \setminus \{0\}.
\end{equation}
From the maximum principle, Lemma \ref{lemma_v1_gen}, and \eqref{v_bar_pointwise_control}, we know that
$$\|v_1\|_{L^\infty(B_R)} = \sup_{x' \in \partial B_R \cup \{0\}} |v_1(x')| \le CR^{1-t}\|\nabla \bar{v}\|_{\varepsilon,-t,1, B_{R_0}}.$$
Therefore, by the boundary gradient estimate,
\begin{equation}\label{grad_v1_boundary}
|\nabla v_1(x')| \le C\|\nabla \bar{v}\|_{\varepsilon,-t,1, B_{R_0}}  R^{-t} \quad \mbox{for}~~x' \in B_{R} \setminus B_{R/2}.
\end{equation}
In particular, $v_1 \in H^1(B_R)$. Therefore, $v_2 \in H^1_0(B_R)$ satisfies
$$
\dv\Big[ \Big( \varepsilon + a\Big(\frac{x'}{|x'|}\Big)|x'|^2 \Big) \nabla  v_2\Big]= \dv F - \varepsilon \Delta v_1\quad\text{in}\,\,B_{R}.
$$
Let $\tilde{v}_1 (y') = v_1(Ry')$, $\tilde{v}_2 (y') = v_2(Ry')$, $\tilde{F}(y') = R^{-1}F(Ry')$, and $\tilde{\varepsilon} = \varepsilon R^{-2}$. Then by \eqref{grad_v1_interior} and \eqref{grad_v1_boundary},
\begin{equation}
            \label{eq2.34}
 \| \tilde F \|_{\tilde \varepsilon,s,0,B_{1}} = R^{1+ s} \|F\|_{\varepsilon,s,0,B_{R}}, \quad \|\nabla \tilde v_1\|_{\tilde \varepsilon,\alpha - 1,1, B_1} \le CR^{1-t} \|\nabla \bar{v}\|_{\varepsilon,-t,1, B_{R_0}},
\end{equation}
and $\tilde v_2$ satifies
\begin{equation}
\label{v2_tilde_equation}
\dv\Big[ \Big( \tilde \varepsilon + a\Big(\frac{x'}{|x'|}\Big)|x'|^2 \Big) \nabla  \tilde v_2\Big]= \dv \tilde F - \tilde \varepsilon \Delta \tilde v_1\quad\text{in}\,\,B_{1}.
\end{equation}
Denote $x' = (r ,\xi)$ in the polar coordinates, where $0 \le r \le 1$ and $\xi \in \bS^{n-2}$. We multiply \eqref{v2_tilde_equation} by $\tilde v_2$ and integrate by parts to get
\begin{align*}
\int_{B_1} (\tilde \varepsilon+ a(\xi)r^2) |\nabla \tilde v_2|^2 = \int_{B_1} \tilde F\cdot \nabla v_2 + \tilde \varepsilon \int_{B_1} \nabla \tilde v_1 \cdot \nabla \tilde v_2.
\end{align*}
By Young's inequality and the definitions of $\| \tilde F\|= \| \tilde F\|_{\varepsilon,s,0,B_{1}}$ and $\|\nabla \tilde v_1\| = \|\nabla \tilde v_1\|_{\tilde \varepsilon,\alpha - 1,1, B_1}$, since $2s + n-2 > -1$,
\begin{align*}
\int_{B_1} (\tilde \varepsilon+ r^2) |\nabla \tilde v_2|^2  &\le C\| \tilde F \|^2 \int_{B_1} r^{2s}(\tilde \varepsilon + r^2) + C \|\nabla \tilde v_1\|^2  \int_{B_1} \frac{\tilde\varepsilon^2 r^{2\alpha - 2}}{\tilde\varepsilon + r^2}\\
& \le C\| \tilde F \|^2 (\tilde{\varepsilon} + 1) + C \|\nabla \tilde v_1\|^2 \left\{
\begin{aligned}
&\tilde\varepsilon^{\alpha + \frac{n-1}{2}}, &&\mbox{when}~~n < 5-2\alpha;\\
&\tilde\varepsilon^{2}(|\ln \tilde\varepsilon|+1), &&\mbox{when}~~n = 5-2\alpha;\\
&\tilde\varepsilon^{2}, &&\mbox{when}~~n > 5-2\alpha,
\end{aligned}
\right.
\end{align*}
By the inequality above and Lemma \ref{calculus_lemma} with $\beta = 0$,
$$
\sup_{0<r<1} r^n \fint_{\partial B_r} |\tilde v_2|^2 \, d\sigma \le C \left( (\tilde \varepsilon + 1) \| \tilde F \|_{\tilde \varepsilon,s,0,B_{1}}^2 +\tilde \varepsilon^{2\tilde{\alpha}}  \|\nabla \tilde v_1\|_{\tilde \varepsilon,\alpha - 1,1, B_1}^2 \right),
$$
which together with \eqref{eq2.34} implies, for any $0 < \rho < R$,
\begin{align}\label{v2_L2_sphere}
&\fint_{\partial B_\rho} |v_2(x')|^2 \, d\sigma \nonumber\\
&\le C \left( \frac{R}{\rho} \right)^{n}\left(R^{2+2s} \left( \frac{\varepsilon}{R^2} + 1 \right) \| F \|^2_{\varepsilon,s,0,B_{R}} + \left( \frac{\varepsilon}{R^2} \right)^{2\tilde\alpha} R^{2-2t} \|\nabla \bar{v}\|^2_{\varepsilon,-t,1, B_{R_0}} \right).
\end{align}
Denote
\begin{equation}
                        \label{eq2.52}
\omega(\rho):= \Big(\fint_{B_\rho \setminus B_{\rho/2}} a\Big(\frac{x'}{|x'|}\Big) |\bar v(x') - (\bar v)^a_{B_\rho \setminus B_{\rho/2}}|^2\,d\sigma\Big)^{1/2}.
\end{equation}
By Lemma \ref{lemma_v1_gen} and \eqref{v2_L2_sphere}, for any $0 < \rho < R/4$,
\begin{align*}
& \fint_{\partial B_\rho}  a\Big(\frac{x'}{|x'|}\Big)|\bar v(x') - v_1(0)|^2 \, d\sigma\\
&\le C \fint_{\partial B_\rho}  a\Big(\frac{x'}{|x'|}\Big)|v_1(x') - v_1(0)|^2 \, d\sigma  + C \fint_{\partial B_\rho}  a\Big(\frac{x'}{|x'|}\Big)|v_2(x')|^2 \, d\sigma \\
&\le C \left(\frac{\rho}{R} \right)^{2\alpha}  \fint_{\partial B_R}  a\Big(\frac{x'}{|x'|}\Big) |\bar v(x') - (\bar v)^a_{\partial B_R}|^2 \, d\sigma \\
&\quad +C \left( \frac{R}{\rho} \right)^{n}\left(R^{2+2s} \left( \frac{\varepsilon}{R^2} + 1 \right) \| F \|^2_{\varepsilon,s,0,B_{R}} + \left( \frac{\varepsilon}{R^2} \right)^{2\tilde\alpha} R^{2-2t} \|\nabla \bar{v}\|^2_{\varepsilon,-t,1, B_{R_0}} \right).
\end{align*}
Multiplying both sides by $\rho^{n-2}$, integrating over $(\rho/2, \rho)$ and dividing both sides by $\rho^{n-1}$, we have for any $0 < \rho < \frac{r}{4} \le \frac{R}{4}$,
\begin{align*}
\omega(\rho)^2 &\le \fint_{B_\rho \setminus B_{\rho/2}}  a\Big(\frac{x'}{|x'|}\Big)|\bar v(x') - v_1(0)|^2\,d\sigma\\
&\le C\left(\frac{\rho}{r} \right)^{2\alpha}  \fint_{\partial B_r}  a\Big(\frac{x'}{|x'|}\Big) |\bar v(x') - (\bar v)^a_{\partial B_r}|^2 \, d\sigma \\
&\quad +C \left( \frac{r}{\rho} \right)^{n}\left(r^{2+2s} \left( \frac{\varepsilon}{r^2} + 1 \right)\| F \|^2_{\varepsilon,s,0,B_{r}} + \left( \frac{\varepsilon}{r^2} \right)^{2\tilde\alpha} r^{2-2t} \|\nabla \bar{v}\|^2_{\varepsilon,-t,1, B_{R_0}} \right)\\
&\le C\left(\frac{\rho}{r} \right)^{2\alpha}  \fint_{\partial B_r}  a\Big(\frac{x'}{|x'|}\Big) |\bar v(x') - (\bar v)^a_{B_R \setminus B_{R/2}}|^2 \, d\sigma\\
&\quad +C \left( \frac{r}{\rho} \right)^{n}\left(r^{2+2s} \left( \frac{\varepsilon}{r^2} + 1 \right) \| F \|^2_{\varepsilon,s,0,B_{R}} + \left( \frac{\varepsilon}{r^2} \right)^{2\tilde\alpha} r^{2-2t} \|\nabla \bar{v}\|^2_{\varepsilon,-t,1, B_{R_0}} \right).
\end{align*}
Multiplying both sides by $r^{n-2}$, integrating $r$ over $(R/2, R)$, and dividing both sides by $R^{n-1}$ give \eqref{iteration_gen}.
\end{proof}

\section{Proof of Theorem \ref{touch_thm}}
In this section, we give the proof of Theorem \ref{touch_thm}. In this case, $\Gamma_+$ and $\Gamma_-$ touch at the origin. After a suitable rotation in $\bR^{n-1}$, we may assume without loss of generality that $D^2(f-g)(0')$ is a diagonal matrix whose entries are denoted by $a_1, a_2, \ldots, a_{n-1} > 0$. Therefore,
\begin{equation}
                \label{eq2.38}
f(x') - g(x') = \sum_{i=1}^{n-1}a_i x_i^2 + e(x'),
\end{equation}
where $e(x')$ satisfies $|e(x')| \le C|x'|^{2 + \gamma}$.

Let $u \in H^1(\Omega_{R_0})$ be a solution of \eqref{main_problem_narrow}. For $x \in \Omega_{R_0}$, we consider, as in \cite{DLY},
\begin{equation}
\label{u_bar_def}
\bar u(x'):= \fint_{g(x')}^{f(x')} u(x', x_n) \, dx_n.
\end{equation}
It follows from a direct computation that $\bar{u} \in H^1(B_{R_0}, |x'|^2dx')$. For any $0 < R \le R_0/2$, we make a change of variables by setting
\begin{equation*}
\left\{
\begin{aligned}
y' &= x' ,\\
y_n &=  2 R^2 \left( \frac{x_n - g(x')}{f(x') - g(x')} - \frac{1}{2} \right),
\end{aligned}\right.
\quad \forall (x',x_n) \in \Omega_{2R} \setminus \Omega_{R}.
\end{equation*}
This change of variables maps the domain $\Omega_{2R} \setminus \Omega_{R}$ to $Q_{2R, R^2} \setminus Q_{R, R^2}$, where
\begin{equation}\label{Q_s_t}
Q_{s,t}:= \{ y = (y',y_n) \in \bR^n ~\big|~  |y'| < s,  |y_n| < t\}
\end{equation}
for $s,t > 0$. Let $v(y) = u(x)$, so that $v(y)$ satisfies
\begin{equation*}
\left\{
\begin{aligned}
-\partial_i(B^{ij}(y) \partial_j v(y)) &=0 \quad \mbox{in } Q_{2R, R^2} \setminus Q_{R , R^2},\\
 B^{nj}(y) \partial_j v(y) &= 0 \quad \mbox{on } \{z_n = -R^2\} \cup \{z_n = R^2\},
\end{aligned}
\right.
\end{equation*}
where
\begin{align*}
&(B^{ij}(y)) = \frac{2R^2(\partial_x y)(A^{ij}(x(y)))(\partial_x y)^t}{\det (\partial_x y)}\\
&=\frac{ 2 R^2(\partial_x y)(\partial_x y)^t}{\det (\partial_x y)} + \frac{ 2 R^2(\partial_x y)(A^{ij}(x(y)) - \delta_{ij})(\partial_x y)^t}{\det (\partial_x y)} =: (C^{ij}(y)) + (D^{ij}(y))
\end{align*}
and
$$
\det (\partial_x y)=2R^2(f(y')-g(y'))^{-1}.
$$
Note that the top left $(n-1)\times(n-1)$ part of $(C^{ij}(y))$ is $(f(y') - g(y'))I_{(n-1)\times(n-1)}$. Let
$$
\bar{v}(y') = \fint_{-R^2}^{R^2} v(y', y_n) \, dy_n(=\bar u(y')).
$$
Then $\bar{v}$ satisfies in $B_{2R} \setminus B_R \subset \bR^{n-1}$ that
$$
\dv\Big[\Big( f(y') - g(y') \Big)\nabla \bar v \Big] = - \sum_{i = 1}^{n-1} \partial_i \overline{C^{in} \partial_n v}
- \sum_{i=1}^{n-1} \sum_{j=1}^n \partial_i \overline{D^{ij} \partial_j v},
$$
where $\overline{h}$ denotes the average of $h$ with respect to $y_n$ in the interval $(-R^2,R^2)$. Reversing the change of variables, one can see that $\bar{u}$ satisfies in $B_{R_0} \setminus\{0\} \subset \bR^{n-1}$ that
$$
\dv\Big[\Big( f(x') - g(x') \Big)\nabla \bar u \Big] = - \sum_{i = 1}^{n-1} \partial_i \overline{b^i \partial_n u} - \sum_{i=1}^{n-1} \sum_{j=1}^n \partial_i \overline{c^{ij} \partial_j u},
$$
where for $1 \le i \le n-1,$
\begin{align*}
b^i(x) &= (f(x') - g(x')) \partial_i g(x') + (x_n - g(x'))\partial_i(f(x') - g(x')),\\
c^{ij}(x) &= \left(A^{ij}(x) - A^{ij}(0) \right) \left(f(x') - g(x') \right)\quad \mbox{for}~~1 \le j \le n-1,\\
c^{in}(x)&= \sum_{k=1}^{n-1} \left(A^{ik}(x) - A^{ik}(0) \right)b^k + \left(A^{in}(x) - A^{in}(0) \right) \left(f(x') - g(x') \right),
\end{align*}
and $\overline{h}$ denotes the average of $h$ with respect to $x_n$ in the interval $(g(x'),f(x'))$ as in \eqref{u_bar_def}. By the weak formulation and $\bar{u} \in H^1(B_{R_0}, |x'|^2dx')$, one can see that $\bar u$ satisfies the above equation in $B_{R_0}$.  Therefore, $\bar u$ satisfies
\begin{equation}
\label{u_bar_equation_2}
\dv\Big[\Big( \sum_{i=1}^{n-1} a_i |x_i|^2 \Big)\nabla \bar u \Big] = \dv F \quad \mbox{in}~~B_{R_0} \subset \bR^{n-1},
\end{equation}
where $F_i = -\overline{b^i \partial_n u} - e \partial_i \bar{u} - \sum_{j=1}^n \overline{c^{ij}\partial_j u}$ and $e$ is given in \eqref{eq2.38}.
From the assumptions \eqref{A_ij_assumption}, \eqref{fg_0}, and \eqref{fg_1}, we have
$$
|b^i(x)| \le C|x'|^3, \quad |c^{ij}(x)| \le C|x'|^{2 + \gamma} \quad \mbox{for}~~i=1,\ldots,n-1, j = 1,\ldots, n.
$$
Hence
\begin{equation*}
|F(x')| \le C |x'|^{2+\gamma}\overline{|\nabla u(x')|} \quad \mbox{for}~~x'\in B_{R_0}.
\end{equation*}

\begin{proof}[Proof of Theorem \ref{touch_thm}]
Without loss of generality, we assume that $\|u\|_{L^\infty(\Omega_{R_0})} = 1$. Let $\bar{u}$ be defined as in \eqref{u_bar_def}. By \eqref{grad_u_bound_rough} with $\varepsilon = 0$,
$$\|\nabla \bar{u}\|_{0,-s_0,1,B_{R_0}} < \infty,$$
where $s_0 = 1$. Then $\bar{u}$ satisfies the equation \eqref{u_bar_equation_2} with $F$ satisfying
$$\|F\|_{0,2+\gamma-s_0, 1, B_{R_0}} < \infty.$$
By \eqref{grad_u_bound_rough} with $\varepsilon = 0$,
\begin{equation}
\label{u-u_bar}
|u(x', x_n) - \bar{u}(x')| \le (f(x')-g(x')) \max_{x_n\in (g(x') , f(x'))} |\partial_n u(x',x_n)| \le C|x'| \quad \mbox{in}~~\Omega_{R_0}.
\end{equation}
By Proposition \ref{prop_grad_u_bar_control} and \eqref{grad_u_bound_rough}, both $u$ and $\bar{u}$ are H\"older continuous. Indeed, for any $x,y\in \Omega_{R_0}$ such that $|x'|\le |y'|$, we denote $r=|x-y|$. When $r\le |x'|^{2}$, by \eqref{grad_u_bound_rough} and the mean value formula, we have
$$
|u(x)-u(y)|\le Cr|x'|^{-1}\le Cr^{1/2}.
$$
When $r>|x'|^2$, by \eqref{u-u_bar} and using the $C^\beta$ regularity of $\bar u$, we also have
\begin{align*}
|u(x)-u(y)|&\le |u(x)-\bar u(x')|+|u(y)-\bar u(y')|+|\bar u(x')-\bar u(y')|\\
&\le C|x'|+Cr^\beta\le C(r^{1/2}+r^\beta).
\end{align*}
Combining the above two estimates, we see the H\"older continuity of $u$.
Thus, we may further assume, without loss of generality, that $u(0) = \bar u(0) = 0$.
By decreasing $\gamma$ if necessary, we may assume that $1+\gamma-s_0=\gamma<\alpha(\lambda_1)$.
By Proposition \ref{prop_grad_u_bar_control} and \eqref{u-u_bar}, we have, for any $0 < R < R_0/4$,
\begin{align*}
\fint_{\Omega_{4R} \setminus \Omega_{R/2}} |u |^2 \, dx &\le C\fint_{\Omega_{4R} \setminus \Omega_{R/2}} |u - \bar{u}|^2 \, dx + C \fint_{\Omega_{4R} \setminus \Omega_{R/2}} |\bar{u} |^2 \, dx \le C R^{2\tilde\alpha},
\end{align*}
where $\tilde{\alpha} = \min\{\alpha(\lambda_1) , 1+ \gamma - s_0 \}$. We make a change of variables by setting
\begin{equation*}
\left\{
\begin{aligned}
z' &= x' ,\\
z_n &=  2 R^2 \left( \frac{x_n - g(x')}{f(x') - g(x')} - \frac{1}{2} \right),
\end{aligned}\right.
\quad \forall (x',x_n) \in \Omega_{4R} \setminus \Omega_{R/2}.
\end{equation*}
This change of variables maps the domain $\Omega_{4R} \setminus \Omega_{R/2}$ to $Q_{4R, R^2} \setminus Q_{R/2 , R^2}$, where $Q_{s,t}$ is defined as \eqref{Q_s_t}. Let $w(z) = u(x)$, so that $w(z)$ satisfies
\begin{equation*}
\left\{
\begin{aligned}
-\partial_i(b^{ij}(z) \partial_j w(z)) &=0 \quad \mbox{in } Q_{4R, R^2} \setminus Q_{R/2 , R^2},\\
 b^{nj}(z) \partial_j w(z) &= 0 \quad \mbox{on } \{z_n = -R^2\} \cup \{z_n = R^2\},
\end{aligned}
\right.
\end{equation*}
where
$$
(b^{ij}(z)) = \frac{(\partial_x z)(A^{ij}(x(z)))(\partial_x z)^t}{\det (\partial_x z)}.
$$
It is straightforward to verify that
$$
\frac{I}{C} \le b(z) \le CI \quad \mbox{and}\quad \|b\|_{C^\gamma (Q_{4R, R^2} \setminus Q_{R/2 , R^2})} \le CR^{-\gamma}.
$$
Let $\tilde{b}^{ij}(z) = b^{ij}(R z)$ and $\tilde{w}(z) = w(R z)$. Then $\tilde{w}$ satisfies
\begin{equation*}
\left\{
\begin{aligned}
-\partial_i(\tilde b^{ij}(z) \partial_j \tilde w(z)) &=0 \quad \mbox{in } Q_{4, R} \setminus Q_{1/2, R},\\
\tilde b^{nj}(z) \partial_j \tilde w(z) &= 0 \quad \mbox{on } \{z_n = -R\} \cup \{z_n = R\},
\end{aligned}
\right.
\end{equation*}
with
$$
\frac{I}{C} \le \tilde{b} \le CI \quad \mbox{and} \quad \| \tilde{b} \|_{C^{\gamma} (Q_{4, R} \setminus Q_{1/2, R})} \le C.
$$
Now we define
$$
S_l:= \left\{z \in \bR^n ~\big|~  1/2 < |z'| < 4,~ (2l-1) R < z_n < (2l+1) R \right\}
$$
for any integer $l$, and
$$
S_{s,t}^m: = \left\{z \in \bR^n ~\big|~  s < |z'| < t,~ |z_n| < m\right\}.
$$
Note that $ Q_{4, R} \setminus Q_{1/2, R} = S_0$. We take the even extension of $\tilde w$ with respect to $y_n=R$ and then take the periodic extension (so that the period is equal to $4R$).
More precisely,
we define, for any $l \in \bZ$, a new function $\hat{w}$ by setting
$$\hat{w}(z) := \tilde{w}\left(z', (-1)^l\left(y_n - 2l R\right)\right), \quad \forall z \in S_l.$$
We also define the corresponding coefficients, for $k = 1,2, \cdots, n-1$,
$$\hat{b}^{nk}(z)=\hat{b}^{kn}(z) := (-1)^l\tilde{b}^{nk}\left(z', (-1)^l\left(z_n - 2l \rho\right)\right),  \quad \forall z \in S_l,$$
and for other indices,
$$\hat{b}^{ij}(z) := \tilde{b}^{ij}\left(z', (-1)^l\left(z_n - 2l \rho\right)\right), \quad \forall y \in S_l.$$
Then $\hat{w}$ and $\hat{c}^{ij}$ are defined in the infinite ring $Q_{4, \infty} \setminus Q_{1/2, \infty}$. In particular, $\hat{w}$ satisfies the equation
$$
\partial_i (\hat{b}^{ij} \partial_j \hat{w}) = 0 \quad \mbox{in}\,\,S_{1/2,4}^2.
$$
By \cite{LN}*{Proposition 4.1} and \cite{LY2}*{Lemma 2.1}, we have
$$
\| \nabla \hat w \|_{L^\infty(S_{1,2}^1)} \le C \| \hat w\|_{L^2(S_{1/2,4}^2)} \le CR^{\tilde{\alpha}},
$$
which, after reversing the changes of variables, implies,
$$
\| \nabla u\|_{L^\infty(\Omega_{2R} \setminus \Omega_R)} \le CR^{\tilde{\alpha}-1}.
$$
Therefore, we have improved the upper bound $|\nabla u(x)| \le C|x'|^{-s_0}$ to $|\nabla u(x)| \le C|x'|^{\tilde{\alpha}-1}$, where $\tilde{\alpha}-1 = \min\left\{\alpha(\lambda_1) - 1, -s_0 + \gamma \right\}$. If $-s_0 + \gamma < \alpha(\lambda_1) - 1$, we take $s_1 = s_0 - \gamma$ and repeat the argument above. We may decrease $\gamma$ if necessary so that $\alpha(\lambda_1) - 1 \neq -s_0 +k\gamma$ for any $k=1,2,\ldots$.
After repeating the argument finitely many times, we obtain the estimate \eqref{main_goal_touch_case}.
\end{proof}

\section{Proof of Theorem \ref{gen_thm}}
In this section, we give the proof of Theorem \ref{gen_thm}. Without loss of generality, we assume $\| u\|_{L^\infty(\Omega_{R_0})} = 1$.
We perform a change of variables by setting
\begin{equation}\label{x_to_y_2}
\left\{
\begin{aligned}
y' &= x' ,\\
y_n &= 2 \varepsilon \left( \frac{x_n - g(x') + \varepsilon/2}{\varepsilon + f(x') - g(x')} - \frac{1}{2} \right),
\end{aligned}\right.
\quad \forall (x',x_n) \in \Omega_{R_0}.
\end{equation}
This change of variables maps the domain $\Omega_{R_0}$ to $Q_{R_0, \varepsilon}$, where $Q_{s,t}$ is defined as in \eqref{Q_s_t}. Moreover,
\begin{equation}
\label{det_d_x_y}
\det(\partial_xy) = 2\varepsilon (\varepsilon + f(x') - g(x'))^{-1}.
\end{equation}
After a suitable rotation in $\bR^{n-1}$, we may assume without loss of generality that $D^2(f-g)(0')$ is a diagonal matrix whose entries are denoted by $a_1, a_2, \ldots, a_{n-1}>0$ and \eqref{eq2.38} holds.
Let $u \in H^1(B_{R_0})$ be a solution of \eqref{main_problem_narrow}, and let $v(y) = u(x)$. Then $v$ satisfies
\begin{equation}\label{equation_v_2}
\left\{
\begin{aligned}
-\partial_i(b^{ij}(y) \partial_j v(y)) &=0 \quad \mbox{in } Q_{R_0, \varepsilon},\\
b^{nj}(y) \partial_j v(y) &= 0 \quad \mbox{on } \{y_n = -\varepsilon\} \cup \{y_n = \varepsilon\}
\end{aligned}
\right.
\end{equation}
with $\|v\|_{L^\infty(Q_{R_0, \varepsilon})} = 1$,
where the matrix $(b^{ij}(y))$ is given by
\begin{align*}
&(b^{ij}(y))\\
&= \frac{ 2 \varepsilon(\partial_x y)(A^{ij}(x(y)))(\partial_x y)^t}{\det (\partial_x y)} = \frac{ 2 \varepsilon(\partial_x y)(\partial_x y)^t}{\det (\partial_x y)} + \frac{ 2 \varepsilon(\partial_x y)(A^{ij}(x(y)) - \delta_{ij})(\partial_x y)^t}{\det (\partial_x y)}\\
&= \begin{pmatrix}
\varepsilon +  \sum_{j=1}^{n-1} a_j y_j^2 &0 &\cdots &0 &b^{1n}\\
0 &\varepsilon + \sum_{j=1}^{n-1} a_j y_j^2 &\cdots &0 &b^{2n}\\
\vdots &\vdots &\ddots &\vdots &\vdots\\
0 &0 &\cdots &\varepsilon + \sum_{j=1}^{n-1} a_j y_j^2 &b^{n-1,n}\\
b^{n1} &b^{n2} &\cdots &b^{n,n-1} & \frac{\sum_{j=1}^{n-1} |b^{jn}|^2 + 4\varepsilon^2}{\varepsilon + f(y') - g(y')}
\end{pmatrix}\\
&\quad+
\begin{pmatrix}
e^1 &0 &\cdots &0 &0\\
0 &e^2 &\cdots &0 &0\\
\vdots &\vdots &\ddots &\vdots &\vdots\\
0 &0 &\cdots & e^{n-1} &0\\
0 &0 &\cdots & 0 &0
\end{pmatrix}
+
\begin{pmatrix}
c^{11} &c^{12} &\cdots &c^{1n}\\
c^{21} &c^{22} &\cdots &c^{2n}\\
\vdots &\vdots &\ddots &\vdots\\
c^{n1} &c^{n2} &\cdots &c^{nn}
\end{pmatrix},
\end{align*}
and for $i = 1, \ldots , n-1$,
\begin{align*}
b^{ni} &= b^{in} = -2\varepsilon \partial_i g(y') - (y_n + \varepsilon)\partial_i(f(y') - g(y')),\\
e^i &= f(y') - g(y') - \sum_{j=1}^{n-1} a_j y_j^2,
\end{align*}
the matrix $\{c^{ij}\}$ is given by
$$
\frac{ 2 \varepsilon(\partial_x y)(A^{ij}(y) - \delta_{ij})(\partial_x y)^t}{\det (\partial_x y)}.
$$
By \eqref{A_ij_assumption}, \eqref{fg_0}, and \eqref{fg_1}, we know for $i = 1, \ldots , n-1$,
\begin{equation}
\label{coefficient_bounds_3}
|b^{ni}(y)| = |b^{in}(y)| \le C \varepsilon |y'|\quad \mbox{and} \quad |e^i(y')| \le C|y'|^{2+\gamma},
\end{equation}
and for $i = 1, \ldots , n-1$, $j = 1, \ldots , n-1$,
\begin{align}
\label{coefficient_bounds_c_ij}
|c^{ij}(y)| &\le C ( \varepsilon+|y'|^2) \left( |y'|^{\gamma}+ ( \varepsilon + |y'|^2)^{\gamma}\right), \nonumber\\
|c^{in}(y)| &\le C \varepsilon\left(|y'|^{\gamma}+ ( \varepsilon + |y'|^2)^{\gamma}\right).
\end{align}
Note that $e^1(y), \ldots, e^{n-1}(y)$ depend only on $y'$ and are independent of $y_n$.
We define
\begin{equation}
\label{v_bar_def}
\bar{v}(y') := \fint_{-\varepsilon}^\varepsilon v(y',y_n)\, dy_n,
\end{equation}
It is straightforward to verify that $\bar{v}$ satisfies in $B_{R_0}\subset \bR^{n-1}$,
\begin{equation}\label{equation_v_bar_4}
\dv \Big[\Big(\varepsilon+\sum_{i=1}^{n-1} a_i y_i^2 \Big)\nabla \bar v \Big]=-\sum_{i=1}^{n-1}\partial_i\overline{b^{in}\partial_n v} - \sum_{i=1}^{n-1} \partial_i(e^i \partial_i \bar v) - \sum_{i=1}^{n-1} \sum_{j=1}^n \partial_i\overline{c^{ij}\partial_j v},
\end{equation}
with $\| \bar{v} \|_{L^\infty(B_{R_0})} \le 1$, where $\overline{b^{in}\partial_n v}$ and $\overline{c^{ij}\partial_j v}$ are the average of $b^{in}\partial_n v$ and $c^{ij}\partial_j v$ with respect to $y_n$ in $(-\varepsilon,\varepsilon)$ as in \eqref{v_bar_def}.

\begin{proof}[Proof of Theorem \ref{gen_thm}]
We make the change of variables \eqref{x_to_y_2}, and let $v(y) = u(x)$. Then $v$ satisfies \eqref{equation_v_2}. Let $\bar{v}$ be defined as in \eqref{v_bar_def}. By \eqref{grad_u_bound_rough},
$$\|\nabla \bar{v}\|_{\varepsilon,-2s_0,1,B_{R_0}} < \infty,$$
where $s_0 = \frac{1}{2}$. Then $\bar{v}$ satisfies the equation \eqref{equation_v_bar_4}, that is
$$
\dv \big[(\varepsilon+ a(\xi)r^2)\nabla \bar v \big]= \dv F  \quad\text{in}\,\,B_{R_0}\subset \bR^{n-1}
$$
with
$$
F_i = - \overline{b^{in}\partial_n v} - e^i \partial_i \bar v - \overline{c^{ij}\partial_j v},\quad i=1,\ldots,n-1.
$$
By \eqref{grad_u_bound_rough} and \eqref{det_d_x_y},
\begin{equation}
\label{grad_n_v_bound}
|\partial_n v| \le C(\varepsilon+|y'|^2)^{1-s_0}/\varepsilon \quad \mbox{and} \quad |\nabla_{y'} v| \le C(\varepsilon + |y'|^2)^{-s_0}\quad \mbox{in}~~Q_{R_0,\varepsilon}.
\end{equation}
Therefore, by \eqref{coefficient_bounds_3} and \eqref{coefficient_bounds_c_ij},
$$\|F\|_{\varepsilon, \gamma- 2s_0,0, B_{R_0}} < \infty.$$
Denote the left-hand side of \eqref{iteration_gen} by $\omega(\rho)$. By Proposition \ref{v_bar_iteration_prop} with $s = \gamma - 2s_0$ and $t = 2s_0$, for $0 < \rho < R/4\le R_0/4$,
$$
\omega(\rho) \le C \left(\frac{\rho}{R} \right)^{\alpha(\lambda_1)} \omega(R) + C \left( \frac{R}{\rho} \right)^{\frac{n}{2}} R^{1-2s_0}\left(R^{\gamma} \left( \frac{\sqrt \varepsilon}{R} + 1 \right) + \left( \frac{\varepsilon}{R^2} \right)^{\tilde\alpha(\lambda_1)} \right),
$$
where $\tilde\alpha(\lambda_1) > \alpha(\lambda_1)$ is given by \eqref{tilde_alpha} and $\omega$ is defined in \eqref{eq2.52}. Fix a $\bar \mu > 0$ satisfying $\bar{\mu} \tilde\alpha(\lambda_1) < \gamma$. For any $0 <\mu < \bar \mu$, $0 < \rho < R/4$, and $\varepsilon^{\frac{1}{2+\mu}} < R/4 \le R_0/4$, we have
$$
\omega(\rho) \le C \left(\frac{\rho}{R} \right)^{\alpha(\lambda_1)} \omega(R) + C \left( \frac{R}{\rho} \right)^{\frac{n}{2}} R^{1-2s_0 + \mu \tilde\alpha(\lambda_1)}.
$$
If $1 - 2s_0 + \mu \tilde\alpha(\lambda_1) < \alpha(\lambda_1)$, by \cite{GiaMar}*{Lemma 5.13},
\begin{equation}
\label{omega_rho_decay}
\omega(\rho) \le C \rho^{1 - 2s_0 + \mu \tilde\alpha(\lambda_1)} \quad \forall~ \varepsilon^{\frac{1}{2+\mu}} \le \rho < R_0.
\end{equation}
By \eqref{grad_n_v_bound},
\begin{equation}
\label{v-v_bar_2}
|v(y', y_n) - \bar{v}(y')| \le 2\varepsilon \max_{y_n\in (- \varepsilon , \varepsilon)} |\partial_n v(y',y_n)| \le C (\varepsilon + |y'|^2)^{1-s_0} \quad \mbox{in}~~Q_{R_0,\varepsilon}.
\end{equation}
Therefore, by \eqref{omega_rho_decay} and \eqref{v-v_bar_2},
\begin{align*}
&\left( \fint_{Q_{\rho ,\varepsilon} \setminus Q_{\rho/2, \varepsilon}} \left| v(y) - (v)_{Q_{\rho ,\varepsilon} \setminus Q_{\rho/2, \varepsilon}}  \right|^2 \, dy \right)^{\frac{1}{2}}\\
&\le \left( \fint_{Q_{\rho ,\varepsilon} \setminus Q_{\rho/2, \varepsilon}}  |v - \bar{v}| ^2 \, dy \right)^{\frac{1}{2}} + \omega(\rho)\\
&\le C \rho^{1 - 2s_0 + \mu \tilde\alpha(\lambda_1)} \quad \forall~ \varepsilon^{\frac{1}{2+\mu}} \le \rho < R_0,
\end{align*}
where
$$
(v)_{Q_{\rho ,\varepsilon} \setminus Q_{\rho/2, \varepsilon}} := \fint_{Q_{\rho ,\varepsilon} \setminus Q_{\rho/2, \varepsilon}} v(y) \, dy.
$$
This implies
\begin{equation}
\label{u-u_average_L2_control}
\left( \fint_{\Omega_{4\rho} \setminus \Omega_{\rho/2}} \left| u(x) - (u)_{\Omega_{4\rho} \setminus \Omega_{\rho/2}}  \right|^2 \, dx \right)^{\frac{1}{2}} \le C \rho^{1 - 2s_0 + \mu \tilde\alpha(\lambda_1)} \quad \forall~ \varepsilon^{\frac{1}{2+\mu}} \le \rho < R_0/4.
\end{equation}
We will show that
\begin{equation}
\label{one_bootstrap}
|\nabla u(x)| \le C(\varepsilon + |x'|^2)^{-\frac{1}{2} + \frac{1- 2s_0 + \mu \tilde\alpha(\lambda_1)}{2 + \mu}} \quad \mbox{for}~~x\in \Omega_{R_0/4}.
\end{equation}
For any $\varepsilon^{\frac{1}{2+\mu}} \le \rho < \frac{R_0}{4}$, we make a change of variables by setting
\begin{equation}\label{x_to_z}
\left\{
\begin{aligned}
z' &= x' ,\\
z_n &=  2 \rho^2 \left( \frac{x_n - g(x') + \varepsilon/2}{\varepsilon + f(x') - g(x')} - \frac{1}{2} \right),
\end{aligned}\right.
\quad \forall (x',x_n) \in \Omega_{4\rho} \setminus \Omega_{\rho/2}.
\end{equation}
This change of variables maps the domain $\Omega_{4\rho} \setminus \Omega_{\rho/2}$ to $Q_{4\rho, \rho^2} \setminus Q_{\rho/2 , \rho^2}$. Let $w(z) = u(x) - (u)_{\Omega_{4\rho} \setminus \Omega_{\rho/2}}$, so that $w(z)$ satisfies
\begin{equation*}
\left\{
\begin{aligned}
-\partial_i(d^{ij}(z) \partial_j w(z)) &=0 \quad \mbox{in } Q_{4\rho, \rho^2} \setminus Q_{\rho/2 , \rho^2},\\
 d^{nj}(z) \partial_j w(z) &= 0 \quad \mbox{on } \{z_n = -\rho^2\} \cup \{z_n = \rho^2\},
\end{aligned}
\right.
\end{equation*}
where
$$
(d^{ij}(z)) = \frac{(\partial_x z)(A^{ij}(x(z)))(\partial_x z)^t}{\det (\partial_x z)}.
$$
Let $\tilde{d}^{ij}(z) = d^{ij}(\rho z)$ and $\tilde{w}(z) = w(\rho z)$. Then $\tilde{w}$ satisfies
\begin{equation*}
\left\{
\begin{aligned}
-\partial_i(\tilde d^{ij}(z) \partial_j \tilde w(z)) &=0 \quad \mbox{in } Q_{4, \rho} \setminus Q_{1/2, \rho},\\
\tilde d^{nj}(z) \partial_j \tilde w(z) &= 0 \quad \mbox{on } \{z_n = -\rho\} \cup \{z_n = \rho\}.
\end{aligned}
\right.
\end{equation*}
It is straightforward to verify that
$$
\frac{I}{C} \le \tilde{d} \le CI \quad \mbox{and} \quad \| \tilde{d} \|_{C^{\gamma} (Q_{4, \rho} \setminus Q_{1/2, \rho})} \le C.
$$
Using the similar ``flipping argument" as in the proof of Theorem \ref{touch_thm}, we have, by \eqref{u-u_average_L2_control},
\begin{equation}
\label{grad_u_rho_outside}
|\nabla u(x)| \le C|x'|^{-2s_0 + \mu \tilde{\alpha}(\lambda_1)} \quad \mbox{for}~~ \varepsilon^{\frac{1}{2+\mu}} \le |x'| <R_0/4
\end{equation}
and
\begin{equation}
\label{osc_u_rho_outside}
\osc_{\Omega_{2\rho} \setminus \Omega_\rho} u \le C \rho^{1-2s_0 + \mu \tilde{\alpha}(\lambda_1)} \quad \mbox{for}~~ \varepsilon^{\frac{1}{2+\mu}} \le \rho <R_0/4.
\end{equation}
By the maximum principle and \eqref{osc_u_rho_outside}, we have
\begin{equation}
\label{osc_u_inside}
\osc_{\Omega_{2\varepsilon^{\frac{1}{2+\mu}}}} u \le C \varepsilon^{\frac{1-2s_0 + \mu \tilde{\alpha}(\lambda_1)}{2+\mu}}.
\end{equation}
For $\varepsilon^{\frac{1}{2}} \le \rho < \frac{1}{2} \varepsilon^{\frac{1}{2+\mu}},$ we consider $u$ in $\Omega_{4\rho} \setminus \Omega_{\rho/2}$. By the change of variables \eqref{x_to_z}, the same ``flipping argument" as above, and \eqref{osc_u_inside}, we have
\begin{equation}
\label{grad_u_middle}
|\nabla u(x)| \le C|x'|^{-1} \varepsilon^{\frac{1-2s_0 + \mu \tilde{\alpha}(\lambda_1)}{2+ \mu}} \quad \mbox{for}~~ \varepsilon^{\frac{1}{2}} \le |x'| < \varepsilon^{\frac{1}{2+\mu}}.
\end{equation}
Finally, we consider $u \in \Omega_{2\sqrt{\varepsilon}}$, and make changes of variables \eqref{x_to_y_2}. By the same ``flipping argument" and \eqref{osc_u_inside}, we have
\begin{equation}
\label{grad_u_inside}
|\nabla u(x)| \le C\varepsilon^{-\frac{1}{2} + \frac{1-2s_0 + \mu \tilde{\alpha}(\lambda_1)}{2+ \mu}}  \quad \mbox{for}~~  |x'| < \varepsilon^{\frac{1}{2}}.
\end{equation}
Therefore, \eqref{one_bootstrap} is concluded from \eqref{grad_u_rho_outside}, \eqref{grad_u_middle}, and \eqref{grad_u_inside}.

We have improved the upper bound of $|\nabla u(x)| \le C(\varepsilon + |x'|^2)^{-s_0}$ to $|\nabla u(x)| \le C(\varepsilon + |x'|^2)^{-s_1}$, where $s_1 = \frac{1}{2} - \frac{1 - 2s_0 + \mu \tilde\alpha(\lambda_1)}{2 + \mu}$. We can repeat the argument with
$$
\|\nabla \bar{v}\|_{\varepsilon,-2s_1,1,B_{R_0}} + \|F\|_{\varepsilon, \gamma- 2s_1,0, B_{R_0}} < \infty.
$$
Let
$$
s_{i+1} = \frac{1}{2} - \frac{1 - 2s_i + \mu \tilde\alpha(\lambda_1)}{2 + \mu},
$$
which is equivalent to
$$
s_{i+1} - \frac{1}{2} = \frac{2}{2+ \mu} \left( s_i - \frac{1}{2} \right) - \frac{\mu \tilde\alpha(\lambda_1)}{2 + \mu}.
$$
Since $s_0 = \frac{1}{2}$, iterating the equation above gives
$$
s_k = \frac{1}{2}- \frac{\mu \tilde{\alpha}(\lambda_1)}{2 + \mu} \sum_{i=0}^{k-1} \left( \frac{2}{2+\mu} \right)^i \quad \forall k \in \bN.
$$
After repeating this argument $k$ times, we have
\begin{equation}
\label{iteration_k_times}
\omega(\rho) \le C \left(\frac{\rho}{R} \right)^{\alpha(\lambda_1)} \omega(R) + C \left( \frac{R}{\rho} \right)^{\frac{n}{2}} R^{1-2s_k + \mu \tilde\alpha(\lambda_1)} \quad \forall~ \varepsilon^{\frac{1}{2+\mu}} \le \rho < \frac{R}{4} \le \frac{R_0}{4},
\end{equation}
provided that
$$
1-2s_{k-1} + \mu \tilde\alpha(\lambda_1) = \mu \tilde\alpha(\lambda_1) \sum_{i=0}^{k-1}\left(\frac{2}{2+\mu}\right)^i < \alpha(\lambda_1).
$$
Since
$$
\mu \tilde\alpha(\lambda_1) \sum_{i=0}^{\infty}\left(\frac{2}{2+\mu}\right)^i = (2 + \mu) \tilde\alpha(\lambda_1) > \alpha(\lambda_1),
$$
there exists a $k \in \bN$ such that
$$
\mu \tilde\alpha(\lambda_1) \sum_{i=0}^{k-1}\left(\frac{2}{2+\mu}\right)^i < \alpha(\lambda_1) \le \mu \tilde\alpha(\lambda_1) \sum_{i=0}^{k}\left(\frac{2}{2+\mu}\right)^i = 1-2s_k + \mu \tilde\alpha(\lambda_1).
$$
For such $k$, \eqref{iteration_k_times} implies that for any $\alpha < \alpha(\lambda_1)$,
$$
\omega(\rho) \le C(\alpha)\rho^{\alpha} \quad \forall~ \varepsilon^{\frac{1}{2+\mu}} \le \rho < \frac{R_0}{4}.
$$
By the same argument of proving \eqref{one_bootstrap}, we can conclude that
$$
|\nabla u(x)| \le C(\varepsilon + |x'|^2)^{-\frac{1}{2} + \frac{\alpha}{2 + \mu}} \quad \mbox{for}~~x\in \Omega_{R_0/4}.
$$
By taking $\mu$ sufficiently small, this concludes the proof.
\end{proof}

\section{Properties of \texorpdfstring{$\lambda_1$}{lambda1} and its corresponding eigenspace}

In this section, we consider the eigenvalue problem \eqref{SL_problem} with $a(\xi) = \xi^t M \xi$ for some positive definite $(n-1) \times (n-1)$ matrix $M$. We study the properties of $\lambda_1$, the first nonzero eigenvalue of \eqref{SL_problem}, and the properties of its corresponding eigenspace. After a suitable rotation in $\bR^{n-1}$, we may assume without loss of generality that
\begin{equation}
\label{M_condition}
(x')^tMx' = \sum_{j=1}^{n-1} a_j x_j^2, \quad a_1 \ge \ldots \ge a_{n-1} > 0.
\end{equation}
Recall that
$$
\bS^{n-2} = \left\{x' = (x_1, \ldots x_{n-1}) \in \bR^{n-1}~\big|~ \sum_{j=1}^{n-1} x_j^2 = 1  \right\}.
$$
First we prove an estimate on $\lambda_1$ under a more general assumption on $a(\xi)$.

\begin{lemma}\label{lambda_1_upper_bound}
For $n \ge 3$, let $\lambda_1$ be the first nonzero eigenvalue of the eigenvalue problem \eqref{SL_problem} with $a(\xi) > 0$ a.e. satisfying $\ln a \in L^\infty(\bS^{n-2})$ and $\int_{\bS^{n-2}} a x_i = 0$ for all $i = 1,\ldots, n-1$. Then $\lambda_1 \le n-2$, and the equality holds if and only if $a$ is constant.
\end{lemma}
\begin{proof}
Since
\begin{equation}\label{spherical_harmonic}
-\Delta_{\bS^{n-2}} x_i = (n-2) x_i \quad \mbox{on}~~\bS^{n-2}
\end{equation}
for $i = 1, \ldots , n-1$, multiplying the above equation by $ax_i$, and integrating over $\bS^{n-2}$, we have, by the identity $x_i \Delta_{\bS^{n-2}}x_i = - |\nabla_{\bS^{n-2}} x_i|^2 + \frac{1}{2} \Delta_{\bS^{n-2}}(x_i^2)$,
\begin{align*}
(n-2)\int_{\bS^{n-2}}ax_i^2 &= - \int_{\bS^{n-2}}ax_i \Delta_{\bS^{n-2}}x_i\nonumber\\
&= \int_{\bS^{n-2}}a |\nabla_{\bS^{n-2}} x_i|^2 - \frac{1}{2}\int_{\bS^{n-2}}a \Delta_{\bS^{n-2}} (x_i^2).
\end{align*}
Summing over $i = 1,\ldots, n-1$, since $\sum_{i=1}^{n-1} x_i = 1$ on $\bS^{n-2}$, we have
$$
(n-2) \sum_{i=1}^{n-1}\int_{\bS^{n-2}}ax_i^2 = \sum_{i=1}^{n-1} \int_{\bS^{n-2}}a |\nabla_{\bS^{n-2}} x_i|^2.
$$
Thus for at least one $i$,
$$
\int_{\bS^{n-2}}a |\nabla_{\bS^{n-2}} x_i|^2 \le (n-2) \int_{\bS^{n-2}}ax_i^2,
$$
which implies $\lambda_1 \le n-2$. If $\lambda_1 = n-2$, then by the Rayleigh quotient formula,
$$
\int_{\bS^{n-2}}a |\nabla_{\bS^{n-2}} x_i|^2 = (n-2) \int_{\bS^{n-2}}ax_i^2
$$
for all $i = 1,\ldots,n-1$. This implies
\begin{equation}\label{spherical_harmonic_2}
-\dv_{\bS^{n-2}}\Big(a\nabla_{\bS^{n-2}} x_i \Big) =(n-2) a x_i \quad \mbox{for}~~i = 1,\ldots,n-1.
\end{equation}
By an orthogonal transformation, we have
$$
-\dv_{\bS^{n-2}}\Big(a\nabla_{\bS^{n-2}} (e\cdot \xi) \Big) =(n-2) a (e\cdot \xi) \quad \mbox{on}~~\bS^{n-2}
$$
for any unit vector $e \in \bR^{n-1}$. Let $\eta \in C^\infty(\bS^{n-2})$. Multiplying the above equation by $\eta$ and integrating over $\bS^{n-2}$, we have
\begin{align*}
(n-2)\int_{\bS^{n-2}} a (e \cdot \xi) \eta &= \int_{\bS^{n-2}} a \nabla_{\bS^{n-2}} (e \cdot \xi) \cdot \nabla_{\bS^{n-2}} \eta\\
&= \int_{\bS^{n-2}} a (e - (e \cdot \xi)\xi) \cdot \nabla_{\bS^{n-2}} \eta
= \int_{\bS^{n-2}} a e \cdot \nabla_{\bS^{n-2}} \eta.
\end{align*}
This implies
$$
\left| \int_{\bS^{n-2}} a e \cdot \nabla_{\bS^{n-2}} \eta \right| \le C \| \eta \|_{L^1(\bS^{n-2})} \quad \forall \eta \in C^\infty(\bS^{n-2}).
$$
Therefore, $ae \in W^{1,\infty}(\bS^{n-2})$ and hence $a \in W^{1,\infty}(\bS^{n-2})$.
Multiplying \eqref{spherical_harmonic} by $a$ and subtracting \eqref{spherical_harmonic_2}, we have
$$
\nabla_{\bS^{n-2}} a \cdot \nabla_{\bS^{n-2}} x_i = 0 \text{ a.e.}\quad \mbox{for}~~i = 1,\ldots,n-1.
$$
Since the span of $\{\nabla_{\bS^{n-2}} x_1,\ldots, \nabla_{\bS^{n-2}} x_{n-1}\}$ is the tangent space of $\bS^{n-2}$ at $x$, we have $\nabla_{\bS^{n-2}} a=0$ a.e.
Therefore, $a$ is constant.
\end{proof}

In the sequel, we will first discuss the case when $n = 3$ and then the case when $n \ge 4$.

\subsection{The case when \texorpdfstring{$n =3$}{}}

We write $x_1 = \cos \theta$ and $x_2 = \sin \theta$, so that \eqref{M_condition} takes the form
\begin{equation}
\label{M_condition_3d}
(x')^tMx' = \sum_{j=1}^{2} a_j x_j^2 = \frac{a_1+a_2}{2} + \frac{a_1 - a_2}{2} \cos (2 \theta), \quad a_1 \ge a_2 > 0.
\end{equation}

\begin{theorem}
\label{3d_eigenvalue_thm}
For $n = 3$, let $\lambda_1$ be the first nonzero eigenvalue of the eigenvalue problem \eqref{SL_problem} with $a(\xi) = \xi^t M \xi$, where $M$ satisfies \eqref{M_condition_3d}. Then $\lambda_1$ is strictly decreasing with respect to $\frac{a_1}{a_2} \in [1 ,\infty)$, and satisfies
$$
\frac{C_1 a_2^{1/2}}{(a_1+a_2)^{1/2}} \le \lambda_1 \le  \frac{C_2 a_2^{1/2}}{(a_1+a_2)^{1/2}} \quad
\mbox{and} \quad \lambda_1 \le \frac{a_1+3a_2}{3a_1 + a_2}
$$
for some positive constants $C_1,C_2$ independent of $M$. Moreover, when $a_1 > a_2$, the eigenspace corresponding to $\lambda_1$ is one dimensional, the corresponding eigenfunctions have exactly two zeros at $\theta = \pi/2$ and $3\pi/2$, and they are odd with respect to $\theta = \pi/2$ and $3\pi/2$.
\end{theorem}

We prove Theorem \ref{3d_eigenvalue_thm} through the following two lemmas. Denote $\beta = \frac{a_1-a_2}{a_1 + a_2}$. Then $\beta \in [0,1)$ and the eigenvalue problem \eqref{SL_problem} becomes
\begin{equation}
\label{eigenvalue_problem_3d}
[(1 + \beta \cos (2\theta)) u'(\theta)]' = - \lambda (1 + \beta \cos (2\theta)) u(\theta) \quad \mbox{on}~~(0, 2\pi),
\end{equation}
with periodic boundary condition. 
When $\beta = 0$, it is easy to see that $\lambda_1 = 1$.
\begin{lemma}\label{lem5.2}
For $\beta \in (0,1)$, consider the eigenvalue problem \eqref{eigenvalue_problem_3d}. If the first nonzero eigenvalue $\lambda_1(\beta)$ is simple, then the eigenfunctions corresponding to  $\lambda_1(\beta)$ must have zeros at $\theta=0, \pi$ or $\theta = \pi/2, 3\pi/2$.
\end{lemma}

\begin{proof}
By \cite{CodLev}*{Theorem 3.1 in Chapter 8}, the problem \eqref{eigenvalue_problem_3d} has eigenvalues $0 = \lambda_0 < \lambda_1 \le \lambda_2 < \lambda_3 \le \lambda_4 < \ldots$. Namely, $\lambda_{2i} < \lambda_{2i+1} \le \lambda_{2i+2}$ for $i = 0,1,2, \ldots$. Moreover, an eigenfunction corresponding to $\lambda_{2i+1}$ or $\lambda_{2i+2}$ must have  exactly $2i+2$ zeros on $[0, 2\pi)$. To conclude the lemma, we only need to construct two solutions $u_1$ and $u_2$ of \eqref{eigenvalue_problem_3d}, whose zeros are at $\theta=0, \pi$ and $\theta = \pi/2, 3\pi/2$, respectively.

First we consider the Dirichlet problem on $(0, \pi)$:
$$
\left\{
\begin{aligned}
\big[(1 + \beta \cos (2\theta)) u'(\theta)\big]' &= - \mu (1 + \beta \cos (2\theta)) u(\theta) \quad \mbox{in}\,\,(0, \pi),\\
u(0) &= u(\pi) = 0.
\end{aligned}
\right.
$$
From the standard Sturm-Liouville theory, the first eigenvalue $\mu_1 > 0$ is simple and there exists an eigenfunction $u_1 > 0$ in $(0,\pi)$. Taking the odd extension of $u_1$, since $\cos (2\theta)$ is even, we know that $u_1$ satisfies \eqref{eigenvalue_problem_3d} on $\bS_1$ with $\lambda = \mu_1$, and $u_1$ only has zeros at $\theta=0, \pi$ on $[0,2\pi)$.

Then we consider the following Dirichlet problem on $(\pi/2, 3\pi/2)$:
$$
\left\{
\begin{aligned}
\big[(1 + \beta \cos (2\theta)) u'(\theta)\big]' &= - \mu (1 + \beta \cos (2\theta)) u(\theta) \quad \mbox{in}\,\,(\pi/2, 3\pi/2),\\
u(\pi/2) &= u(3\pi/2) = 0.
\end{aligned}
\right.
$$
Let $v(\theta) = u(\theta + \pi/2)$. Then $v$ satisfies
$$
\left\{
\begin{aligned}
\big[(1 - \beta \cos (2\theta)) v'(\theta)\big]' &= - \mu (1 - \beta \cos (2\theta)) v(\theta) \quad \mbox{in}\,\,(0, \pi),\\
v(0) &= v(\pi) = 0.
\end{aligned}
\right.
$$
By the same argument as above, we know that there exist $u_2$ and the first eigenvalue $\mu_2 > 0$, such that $u_2$ satisfies \eqref{eigenvalue_problem_3d} with $\lambda = \mu_2$, and $u_2$ only has zeros at $\theta = \pi/2, 3\pi/2$ on $[0,2\pi)$. Therefore, $\{\mu_1, \mu_2\} = \{ \lambda_1, \lambda_2 \}$ and $u_1, u_2$ are eigenfunctions corresponding to $\mu_1, \mu_2$ respectively.
\end{proof}

As a consequence of this lemma, the problem \eqref{eigenvalue_problem_3d} can be reduced to the following Dirichlet problem in case of studying the first nonzero eigenvalue:
\begin{equation}
\label{eigenvalue_problem_3d_2}
\left\{
\begin{aligned}
\big[(1 + \tilde\beta \cos (2\theta)) u'(\theta)\big]' &= - \mu (1 + \tilde\beta \cos (2\theta)) u(\theta) \quad \mbox{in}\,\,(0, \pi),\\
u(0) &= u(\pi) = 0,
\end{aligned}
\right.
\end{equation}
with $\tilde\beta \in (-1,1]$. Denote $\mu_1(\tilde\beta)$ to be the first  eigenvalue of \eqref{eigenvalue_problem_3d_2}, which is given by the follow Rayleigh quotient
\begin{equation}
\label{Rayleigh_quotient_3d}
\mu_1 (\tilde\beta) = \inf_{u>0 \in H_0^1((0,\pi))}\frac{\int_0^{\pi} (1 + {\tilde\beta} \cos (2\theta)) u'(\theta)^2 \, d\theta}{\int_0^{\pi}(1 + {\tilde\beta} \cos (2\theta)) u(\theta)^2 \, d\theta}.
\end{equation}

\begin{lemma}\label{lem5.3}
Consider the eigenvalue problem \eqref{eigenvalue_problem_3d_2} and let $\mu_1(\tilde{\beta})$ be as above. The function $\mu_1(\tilde\beta)$ is strictly increasing with respect to $\tilde\beta \in (-1,1]$, $\mu_1(1)=3$, and $\lim_{\tilde\beta \to -1} \mu_1(\tilde\beta)=0$. Moreover, we have
\begin{equation}
\label{mu_1_beta}
C_1(1+\tilde\beta)^{1/2}\le \mu_1(\tilde\beta)\le  C_2(1+\tilde\beta)^{1/2} \quad \mbox{and} \quad  \mu_1(\tilde\beta)\le \frac{2+{\tilde\beta}}{2-{\tilde\beta}}
\end{equation}
for some constants $C_1,C_2>0$ independent of $\tilde\beta$.
\end{lemma}

\begin{proof}
First, suppose that $u_{\tilde\beta}$ is an eigenfunction corresponding to $\mu_1(\tilde\beta)$, which is positive on $(0,\pi)$. Since $\cos(2\theta)=\cos(2(\pi-\theta))$, it is easily seen that $u_\alpha(\pi-\cdot)$ is also an eigenfunction. Therefore, $u_{\tilde\beta}(\pi-\cdot)$ is a multiple of $u_{\tilde\beta}$. Because $\max u_{\tilde\beta}(\pi-\cdot)=\max u_{\tilde\beta}$ and both are nonnegative, we get $u_{\tilde\beta}(\pi-\cdot)=u_{\tilde\beta}$. This implies that $u_{\tilde\beta}$ can be written as an expansion of $\sin(k\theta),k=1,3,5,\ldots$ on $[0,\pi]$.

We define
\begin{align*}
A_{\tilde\beta}:&= \int_0^\pi |u'_{\tilde\beta}|^2 - \mu_1(\tilde\beta) \int_0^\pi |u_{\tilde\beta}|^2,\\
B_{\tilde\beta}:&= \int_0^\pi \cos(2\theta) |u'_{\tilde\beta}|^2 - \mu_1(\tilde\beta) \int_0^\pi \cos(2\theta) |u_{\tilde\beta}|^2.
\end{align*}
Because $u_{\tilde\beta}$ is a solution of \eqref{eigenvalue_problem_3d_2} with $\mu = \mu_1(\tilde\beta)$, we have $A_{\tilde\beta} = -{\tilde\beta} B_{\tilde\beta}$.
For ${\tilde\beta}\in (-1,1]$, by taking $u=\sin(\theta)$ in the Rayleigh quotient \eqref{Rayleigh_quotient_3d}, we see that
\begin{equation*}
\mu_1({\tilde\beta})\le (2+{\tilde\beta})/(2-{\tilde\beta}).
\end{equation*}
This concludes the second inequality in \eqref{mu_1_beta}. When ${\tilde\beta} = 1$, $u=\sin(\theta)$ is a solution to \eqref{eigenvalue_problem_3d_2} with $\mu=3$. Since $\sin(\theta)$ is strictly positive on $(0,\pi)$, we infer that $\mu_1(1)=3$.
Thus $\mu_1({\tilde\beta})< \mu_1(1) = 3$, and
$$
\int_0^{\pi} (1 + \cos (2\theta)) |u'_{\tilde\beta}|^2 \ge 3 \int_0^{\pi}(1 + \cos (2\theta)) |u_{\tilde\beta}|^2 > \mu_1({\tilde\beta}) \int_0^{\pi}(1 + \cos (2\theta)) |u_{\tilde\beta}|^2.
$$
Namely, $A_{\tilde\beta} + B_{\tilde\beta} > 0$. Therefore, $(1 - {\tilde\beta}) B_{\tilde\beta} > 0$, which implies $B_{\tilde\beta} > 0$ for any ${\tilde\beta} \in (-1,1)$. For any $-1 < {\tilde\beta}_1 < {\tilde\beta} < 1$, we have $A_{\tilde\beta} = - {\tilde\beta} B_{\tilde\beta} < -{\tilde\beta}_1 B_{\tilde\beta}$. Namely,
$$
\int_0^{\pi} (1 + {\tilde\beta}_1 \cos (2\theta)) |u'_{\tilde\beta}|^2 < \mu_1({\tilde\beta}) \int_0^{\pi} (1 + {\tilde\beta}_1 \cos (2\theta)) |u_{\tilde\beta}|^2.
$$
Therefore, $\mu_1({\tilde\beta}_1) < \mu_1({\tilde\beta})$.
Next, we show the first inequality in \eqref{mu_1_beta}. We may certainly assume that ${\tilde\beta}\in (-1,-1/2)$. Let $\varepsilon=1+{\tilde\beta}\in (0,1/2)$. Define $u(\theta)=\varepsilon^{-1/2}\theta$ when $\theta\in [0,\varepsilon^{1/2}]$, $u(\theta)=1$ when $\theta\in (\varepsilon^{1/2},\pi-\varepsilon^{1/2})$, and $u(\theta)=\varepsilon^{-1/2}(\pi-\theta)$ when $\theta\in [\pi-\varepsilon^{1/2},\pi]$. Then
\begin{align*}
&\int_0^\pi (1+{\tilde\beta}\cos(2\theta))|u'|^2\,d\theta
=2\int_0^{\sqrt\varepsilon}(1+{\tilde\beta}\cos(2\theta))\varepsilon^{-1}\,d\theta\notag\\
&=2\int_0^{\sqrt\varepsilon}(\varepsilon-2{\tilde\beta}\sin^2\theta)
\varepsilon^{-1}\,d\theta \le C\int_0^{\sqrt\varepsilon}(1-{\tilde\beta}\theta^2 \varepsilon^{-1})\,d\theta\le C\sqrt\varepsilon.
\end{align*}
This together with the obvious inequality
$$
\int_0^\pi (1+{\tilde\beta}\cos(2\theta))|u|^2\,d\theta\ge C
$$
and \eqref{Rayleigh_quotient_3d} imply the upper bound of first inequality in \eqref{mu_1_beta}. To see the lower bound, without loss of generality, we assume that
$$
u_{\tilde\beta}(\theta_0)=\max_{\theta\in [0,\pi]} u_{\tilde\beta}(\theta)=1.
$$
By symmetry, we may also assume that $\theta_0\le \pi/2$.
Then by H\"older's inequality,
\begin{align}
1&=u^2_{\tilde\beta}(\theta_0)
\le \Big(\int_0^{\theta_0}|u_{\tilde\beta}'|\,d\theta\Big)^2\notag\\
                    \label{eq7.12}
&\le \Big(\int_0^{\theta_0}(1+{\tilde\beta}\cos(2\theta))|u_{\tilde\beta}'|^2\,d\theta\Big)
\Big(\int_0^{\theta_0}(1+{\tilde\beta}\cos(2\theta))^{-1}\,d\theta\Big).
\end{align}
Note that
\begin{align*}
&\int_0^{\theta_0}(1+{\tilde\beta}\cos(2\theta))^{-1}\,d\theta
=\int_0^{\theta_0}(\varepsilon-2{\tilde\beta}\sin^2\theta)^{-1}\,d\theta\\
&\le C\int_0^{\sqrt\varepsilon}\varepsilon^{-1}\,d\theta
+C\int_{\sqrt\varepsilon}^{\theta_0}\theta^{-2}\,d\theta\le C\varepsilon^{-1/2}.
\end{align*}
Thus from \eqref{eq7.12}, we get
$$
\int_0^{\theta_0}(1+{\tilde\beta}\cos(2\theta))|u_{\tilde\beta}'|^2\,d\theta\ge C\varepsilon^{1/2},
$$
which together with the obvious inequality
$$
\int_0^\pi (1+{\tilde\beta}\cos(2\theta))|u|^2\,d\theta\le C
$$
and \eqref{Rayleigh_quotient_3d} imply the lower bound of first inequality in \eqref{mu_1_beta}. Finally, from \eqref{mu_1_beta}, we conclude that $\lim_{\tilde\beta \to -1} \mu_1(\tilde\beta)=0$.
The lemma is proved.
\end{proof}

\begin{proof}[Proof of Theorem \ref{3d_eigenvalue_thm}]
Let $0 < \lambda_1(\beta) \le \lambda_2(\beta)$ denote the first and the second nonzero eigenvalue of the problem \eqref{eigenvalue_problem_3d}, respectively. By Lemma \ref{lem5.3}, $\mu_1(\tilde{\beta})$ is strictly increasing in $(-1,1]$, so we know that $\lambda_1(\beta) = \mu_1(-\beta)$, $\lambda_2(\beta) = \mu_1(\beta)$. Therefore, $\lambda_1$ being strictly decreasing with respect to $\frac{a_1}{a_2} \in [1 ,\infty)$ follows from the monotonicity of $\mu_1(\tilde{\beta})$ for $\tilde{\beta} \in (-1,0)$. The inequalities in Theorem \ref{3d_eigenvalue_thm} follow from \eqref{mu_1_beta} with $\tilde\beta = -\beta = - \frac{a_1-a_2}{a_1+a_2}$. Finally, when $\beta > 0$, we can see from the proof of Lemma \ref{lem5.2} that the eigenspace corresponding to $\lambda_1$ is one dimensional, the corresponding eigenfunctions have exactly two zeros at $\theta = \pi/2$ and $3\pi/2$, and they are odd with respect to $\theta = \pi/2$ and $3\pi/2$.
\end{proof}

\subsection{Higher dimensional case}

In this subsection, we consider the case when $n\ge 4$.
We will show that there exists a small constant $\varepsilon_0$, depending only on $n$, such that if $$
(1-\varepsilon_0) \frac{I}{\|I\|} \le \frac{M}{\|M\|} \le (1+\varepsilon_0) \frac{I}{\|I\|},
$$the eigenspace corresponding to the first nonzero eigenvalue $\lambda_1$ of \eqref{SL_problem} satisfies the property $O$, which is defined as follows.

\begin{definition}\label{property_O}
We say that a function space on $\bS^{n-2} \subset \bR^{n-1}$ satisfies the property $O$ if it is the span of functions which are odd in one of the $x_i$ variables and even with respect to other variables.
\end{definition}

Indeed, we consider the following operator on $\bS^{n-2}$:
$$
L_\mu=-\text{div}_{\bS^{n-2}}\big((1+\mu b(x))\nabla_{\bS^{n-2}}\big) \quad \mbox{for}~~\mu \in \bR,
$$
where $b(x) \in L^\infty(\bS^{n-2})$. Let $\lambda_{1,\mu}$ and $V_{1,\mu}$ be the corresponding first nonzero eigenvalue and the eigenspace of the eigenvalue problem
$$
L_\mu u = \lambda (1+\mu b(x))u.
$$

\begin{proposition}
\label{property_O_prop}
Consider the above eigenvalue problem,
and assume that $b(x)$ is even with respect all variables, and $V_{1,\mu_0}$ satisfies the property $O$ for some $\mu_0 \in \bR$. Then there exists an $\varepsilon_0$, depending only on $n$,  an upper bound of $\|b\|_{L^\infty}$, and $\mu_0$, such that $V_{1,\mu}$ also satisfies the property $O$ for any $\mu \in (\mu_0 - \varepsilon_0, \mu_0 + \varepsilon_0)$.
\end{proposition}

\begin{proof}
Suppose that an orthogonal basis of $V_{1,\mu_0}$ is given by $\{f_1,\ldots,f_m\}$ with $m\in \{1,\ldots,n-1\}$, where for $j=1,\dots,m$, $f_j$ is odd in $x_j$ and even in other variables. Let $\varepsilon_0>0$ be a small constant to be specified later. The perturbation argument below gives all the eigenfunctions of $L_\mu$ close to $\lambda_{1,\mu_0}$ when $\mu$ is in a small neighborhood of $\mu_0$.

For any $\mu\in (\mu_0-\varepsilon_0,\mu_0+\varepsilon_0)$, we consider the expansions
\begin{equation}
                \label{eq10.31}
\tilde\lambda_1=\lambda_{1,\mu_0}+\sum_{k=1}^\infty \varepsilon^k c_k,
\quad \tilde f_1=f_1+\sum_{k=1}^\infty \varepsilon^k v_k,
\end{equation}
where $\varepsilon=\mu-\mu_0\in (-\varepsilon_0,\varepsilon_0)$.
Then
$$
L_{\mu} \tilde f_1=\tilde \lambda_1 (1+\mu b(x))\tilde f_1
$$ is equivalent to
\begin{align}
                                    \label{eq9.41}
&\Big[\Big(\lambda_{1,\mu_0}+\sum_{k=1}^\infty \varepsilon^k c_k\Big)\Big(1+(\mu_0+\varepsilon)b(x)\Big)-L_{\mu_0}\Big]
\Big(f_1+\sum_{k=1}^\infty \varepsilon^k v_k\Big)\notag\\
&=-\varepsilon\text{div}_{\bS^{n-2}}\Big( b(x)\nabla_{\bS^{n-2}}
\Big(f_1+\sum_{k=1}^\infty \varepsilon^k v_k\Big)\Big).
\end{align}
To solve for $c_k,v_k,k=1,\ldots$, we compare the coefficients of $\varepsilon^k$ on both sides of \eqref{eq9.41}. The zeroth order term on the left-hand side is equal to zero because $f_1$ is an eigenfunction of $L_{\mu_0}$ with the eigenvalue $\lambda_{1,\mu_0}$.

Considering the first order terms, we get
\begin{align}
                \label{eq9.52}
&\big(\lambda_{1,\mu_0}(1+\mu_0 b(x))-L_{\mu_0}\big)v_1\notag\\
&=-(\lambda_{1,\mu_0} b(x)+c_1(1+\mu_0 b(x))) f_1-\text{div}_{\bS^{n-2}}\big( b(x)\nabla_{\bS^{n-2}}f_1\big).
\end{align}
Let $X_0$ and $X_2$ be the orthogonal (complement) spaces of $V_{1,\mu}$ in $L_2(\bS^{n-2})$ and $H^2(\bS^{n-2})$, respectively. Then because  $\lambda_{1,\mu_0}(1+\mu_0 b(x))-L_{\mu_0}$ is self-adjoint, it is easily seen that
$$
(\lambda_{1,\mu_0}(1+\mu_0 b(x))-L_{\mu_0})X_2\subset X_0.
$$
Moreover, the mapping $\lambda_{1,\mu_0}(1+\mu_0 b(x))-L_{\mu_0}:X_2\to X_0$ is injective. By the Fredholm theorem, we also know that the mapping is surjective. Therefore, $\lambda_{1,\mu_0}(1+\mu_0 b(x))-L_{\mu_0}$ has a bounded inverse $\cR$ from $X_0$ to $X_2$, and \eqref{eq9.52} has a unique solution $v_1\in X_2$ if and only if its right-hand side is orthogonal to $f_1,\ldots,f_m$. Since the right-hand side is odd in $x_1$ and even in the other variables, we know that it is orthogonal to $f_2,\ldots,f_m$. To make it to be also orthogonal to $f_1$, we find a unique $c_1$ given by
$$
c_1=\frac{\int_{\bS^{n-2}} \Big(b(x)|\nabla_{\bS^{n-2}}f_1|^2-\lambda_{1,\mu_0}b(x)f_1^2\Big)}
{\int_{\bS^{n-2}}(1+\mu_0b(x))f_1^2}.
$$
Then
$$
v_1=-\cR\Big((\lambda_{1,\mu_0} b(x)+c_1(1+\mu_0 b(x))) f_1+\text{div}_{\bS^{n-2}}\big( b(x)\nabla_{\bS^{n-2}}f_1\big)\Big)\in X_2,
$$
which is odd in $x_1$ and even in other variables because $b(x)$ is even with respect to all variables.

Now considering the second order terms, we get
\begin{align*}
&(\lambda_{1,\mu_0}(1+\mu_0 b(x))-L_{\mu_0})v_2
=-\big(c_2(1+\mu_0 b(x))+c_1b(x) \big)f_1\notag\\
&\quad -(\lambda_{1,\mu_0} b(x)+c_1(1+\mu_0 b(x)))v_1-\text{div}_{\bS_{n-2}}\big( b(x)\nabla_{\bS_{n-2}}v_1\big).
\end{align*}
As before, the right-hand side above is orthogonal to $f_2,\ldots,f_m$. To make it to be also orthogonal to $f_1$, we find a unique $c_2$ given by
$$
c_2=\frac{\int_{\bS_{n-2}} \Big(b\nabla_{\bS_{n-2}}v_1\cdot \nabla_{\bS_{n-2}}f_1-(\lambda_{1,\mu_0} b+c_1(1+\mu_0 b))v_1f_1-c_1 b f_1^2\Big)}{\int_{\bS_{n-2}}(1+\mu_0 b)f_1^2}.
$$
Then
\begin{align*}
v_2=-\cR\Big(&\big(c_2(1+\mu_0 b(x))+c_1b(x) \big)f_1\notag\\
&\quad +(\lambda_{1,\mu_0} b(x)+c_1(1+\mu_0 b(x)))v_1
+\text{div}_{\bS_{n-2}}\big( b(x)\nabla_{\bS_{n-2}}v_1\big)\Big)\in X_2,
\end{align*}
which is odd in $x_1$ and even in other variables.

We can repeat this procedure and solve all the $c_k$ and $v_k$'s inductively. Moreover, all the $v_k$'s are odd in $x_1$ and even in other variables. Note that $|c_k|$ and the $H^2$ norm of $v_k$ can be bounded by $C^k$, where $C$ is some positive constant depending only on $f_1$, $b(x)$, and the norm of $\cR$. Therefore,  by taking $\varepsilon_0$ sufficiently small (with the same dependence), both series in \eqref{eq10.31} are convergent in $\bR$ and $H^2(\bS^{n-2})$ respectively, and $\tilde f_1$ is odd in $x_1$ and even in other variables. Similarly, we can find the eigenpairs $(\tilde \lambda_j,\tilde f_j)$ for $j=2,\ldots,m$. From the min-max formula of the eigenvalues, we know that every eigenvalue is Lipschitz in $\mu$. In particular, for $\varepsilon_0$ sufficiently small, we know that $\lambda_{1,\mu}=\min\{\tilde \lambda_1,\ldots,\tilde \lambda_m\}$ and $V_{1,\mu}$ is spanned by $\tilde f_j$'s for those $j$ such that $\tilde \lambda_j=\lambda_{1,\mu}$. Therefore, $V_{1,\mu}$ satisfies the property $O$.
\end{proof}

Applying Proposition \ref{property_O_prop} with $n \ge 4$, $\|b\|_{L^\infty} \le 8$, $\mu_0 = 0$, we obtain an $\varepsilon_0$ such that $V_{1,\mu}$ also satisfies the property $O$ for any $\mu \in (- \varepsilon_0, \varepsilon_0)$. Setting
$$
b(x) = \sum_{i=1}^{n-2} \frac{a_i - a_{n-1}}{\varepsilon_0 a_{n-1}} 2x_i^2\quad \text{and}\quad
\mu = \frac{\varepsilon_0}{2},
$$
we have 
$$
1 + \mu b(x) = \frac{1}{a_{n-1}} \sum_{i=1}^{n-1} a_i x_i^2.
$$
Therefore, the following corollary follows.

\begin{corollary}
\label{property_o_corollary}
For $n \ge 4$, there exists a small constant $\varepsilon_0$ depending only on $n$, such that if $$
(1-\varepsilon_0) \frac{I}{\|I\|} \le \frac{M}{\|M\|} \le (1+\varepsilon_0) \frac{I}{\|I\|},
$$
the eigenspace corresponding to the first nonzero eigenvalue $\lambda_1$ of \eqref{SL_problem} with $a(\xi) = \xi^t M \xi$ satisfies the property $O$.
\end{corollary}

The following lemma will not be used in this paper.

\begin{lemma}
                \label{lem4.4}
For $n \ge 4$, let $\lambda_1$ be the first nonzero eigenvalue to the eigenvalue problem \eqref{SL_problem} with $a(\xi) = \xi^t M \xi$, where $M$ satisfies \eqref{M_condition}. Then for any $a_2 \ge \ldots \ge a_{n-1} > 0$, we have $\lambda_1\to 0$ as $a_1\to +\infty$.
\end{lemma}

\begin{proof}
We consider in the spherical coordinate: for $x'\in \bS^{n-2} \subset \bR^{n-1}$, we can write
$$
x_1=\cos\theta_1,\quad x_2=\sin\theta_1\cos\theta_2,\quad
x_3=\sin\theta_1\sin\theta_2\cos\theta_3,\ldots,
$$
$$
x_{n-2}=\sin\theta_1\sin\theta_2\cdots\sin\theta_{n-3}\cos\theta_{n-2},\quad
x_{n-1}=\sin\theta_1\sin\theta_2\cdots\sin\theta_{n-3}\sin\theta_{n-2},
$$
where $\theta_1,\theta_2,\ldots,\theta_{n-3}\in [0,\pi]$ and $\theta_{n-2}\in [0,2\pi)$.
Then the proof is similar to that of the first upper bound in \eqref{mu_1_beta} by considering $u(\theta_1)=\max\{-1,\min\{1,\varepsilon^{-1}(\theta_1-\pi/2)\}\}$.
\end{proof}

\section{Proof of Theorem \ref{optimality_touch_case_thm}}

\begin{proof}[Proof of Theorem \ref{optimality_touch_case_thm}]
In this proof, we denote $\alpha = \alpha(\lambda_1)$ for simplicity. After a suitable rotation in $\bR^{n-1}$, we may assume without loss of generality that
$$
(f-g)(x') = \sum_{j=1}^{n-1} a_j x_j^2 + {green}e(x'),
$$
where $|e(x')| \le C|x'|^4$.

\textbf{Step 1.} Since $\widetilde\Omega$ is symmetric in $x_i$ and  $\varphi$ is odd in $x_j$, by the uniqueness of solutions, we know that $u$ is odd in $x_j$. In $\Omega_1$, where $\Omega_r$ is defined as in \eqref{domain_def_Omega}, let $\bar{u}$ be defined as \eqref{u_bar_def}.
By a similar argument as in Section 3 and Theorem \ref{touch_thm}, we know that $\bar{u}$ satisfies
$$\dv\Big[\Big( \sum_{i=1}^{n-1} a_i |x_i|^2 \Big)\nabla \bar u \Big] = \dv F \quad \mbox{in}\,\,B_1 \subset \bR^{n-1},$$
where $F$ satisfies
\begin{equation}
\label{F_bound_2}
|F(x')| \le C|x'|^{2+\alpha} \quad \mbox{for}~~x'\in B_{1},
\end{equation}
$C$ is a positive constant depending only on $n$ and  upper bounds of $\|\partial D_1 \|_{C^{4}}$ and $\|\partial D_2\|_{C^{4}}$.
Again, we denote $Y_{k,i}$ to be a normalized eigenfunction corresponding to $(k+1)$-th eigenvalue $\lambda_k$ of the problem \eqref{SL_problem}, so that $\{Y_{k,i}\}_{k,i}$ forms an orthonormal basis of $L^2(\bS^{1})$ under the inner product \eqref{inner_product}. By the assumption, we denote $Y_{1,j}$ to be the eigenfunction that is odd in $x_j$. It is easily seen that $Y_{1,j}$ is an eigenfunction corresponding to $\lambda_1$ in the half sphere $\bS^{n-2} \cap \{x_j > 0 \}$ with zero Dirichlet boundary condition. Since $\lambda_1$ is the first nonzero eigenvalue of the eigenvalue problem in the sphere, it must be the first eigenvalue of the eigenvalue problem in the half sphere. Therefore, it is simple and $Y_{1,j}$ does not change its sign in the half sphere. Without loss of generality, we assume $Y_{1,j}$ is positive in $\{x_j >0\}$ and negative in $\{x_j < 0\}$. Since $\bar u$ is odd with respect to $x_j = 0$, and in particular $\bar{u}(0) = 0$, we have the following decomposition
\begin{equation}
\label{u_bar_expansion}
\bar{u}(x') = \sum_{k=1}^\infty \sum_{i=1}^{N(k)} U_{k,i}(r)Y_{k,i}(\xi), \quad x' \in B_1\setminus\{0\},
\end{equation}
where $U_{k,i}(r) = \fint_{\bS^{1}} a(\xi) \bar{u}(r,\xi) Y_{k,i}(\xi) \, d\xi$ and $U_{k,i} \in C([0,1)) \cap C^\infty((0,1))$. Then $U_{1,j}$ satisfies $U_{1,j}(0) = 0$ and
\begin{equation*}
LU_{1,j}:= U_{1,j}''(r) + \frac{n}{r} U_{1,j}'(r) - \frac{\lambda_1}{r^2} U_{1,j}(r) = H(r), \quad 0 < r <1,
\end{equation*}
where
\begin{align*}
H(r) &= \int_{\bS^{n-2}} \frac{(\dv F) Y_{1,1}(\xi)}{a(\xi)r^2} \, d\xi = \int_{\bS^{n-2}} \frac{\partial_r F_r + \frac{1}{r} \nabla_\xi F_\xi}{a(\xi) r^2} Y_{1,1}(\xi) \, d\xi\\
&= \partial_r \left(\int_{\bS^{n-2}} \frac{F_r}{a(\xi)r^2} Y_{1,1}(\xi) \, d\xi \right) + \int_{\bS^{n-2}} \frac{2F_r Y_{1,1}}{a(\xi)r^3} - \frac{F_\xi}{r^3} \nabla_\xi \left( \frac{Y_{1,1}(\xi)}{a(\xi)} \right)\, d\xi\\
&=: A'(r) + B(r), \quad 0 < r < 1,
\end{align*}
and $A(r),B(r) \in C^1([0,1))$ satisfy, in view of \eqref{F_bound_2}, that
\begin{equation}
\label{AB_bounds}
|A(r)| \le C(n)r^{\alpha}, \quad |B(r)| \le C(n)r^{\alpha-1}, \quad 0 < r < 1.
\end{equation}

\textbf{Step 2.}
We will prove, for some constant $C_1$, that
\begin{equation}
\label{U_11_formula}
U_{1,j}(r) = C_1 r^\alpha + v(r), \quad 0<r<1,
\end{equation}
where $|v(r)| \le Cr^{1+\alpha}$.
We use the method of reduction of order to find a bounded solution $v$ satisfying $Lv = H$ in $(0,1)$, and then show that  $|v(r)| \le Cr^{1+\alpha}$. Note that $h = r^{\alpha}$ is a solution of $Lh = 0$. Let $v = hw$ and
$$
w(r) := \int_0^r\frac{1}{s^{n+2\alpha}} \int_0^s \tau^{n+\alpha} H(\tau) \, d\tau ds, \quad 0 < r < 1.
$$
By a direct computation,
$$
Lv = L(hw) = h w'' + \left( 2h' + \frac{n}{r} h \right)w' =  H
$$
By \eqref{AB_bounds}, we can estimate $|w(r)| \le Cr$. Therefore, $|v(r)| \le Cr^{1+\alpha}$.
Since $U_{1,j} - v$ is bounded and satisfies $L(U_{1,j} - v) = 0$ in $(0,1)$, we know that $U_{1,j} = C_1h + v$ and \eqref{U_11_formula} follows.

\textbf{Step 3.} Completion of the proof.

Since $D_1$ and $D_2$ are strictly convex and symmetric in $x_1,\ldots,x_{n-1}$, it is easy to see that $\partial_\nu x_j \ge 0$ in $\{x_j \ge 0\}$ and $\partial_\nu x_j \le 0$ in $\{x_j \le 0\}$. Therefore, under the assumptions of Theorem \ref{optimality_touch_case_thm}, $x_j$ is a subsolution of \eqref{general_equation} in $\{x_j \ge 0\}$, and is a supersolution of \eqref{general_equation} in $\{x_j \le 0\}$. Hence, $u \ge x_j$ in $\{x_j \ge 0\}$ and $u \le x_j$ in $\{x_j \le 0\}$. Then, $|\bar{u}(x')| \ge |x_j|$ in $B_1 \subset \bR^{n-1}$. Since $Y_{1,j}$ has the same sign as $x_j$, we have
$$
U_{1,j} = \fint_{\bS^{n-2}} a(\xi) \bar{u}(r,\xi) Y_{1,j}(\xi) \, d\xi \ge Cr
$$
for some positive constant $C$. This implies $C_1 > 0$. By \eqref{u_bar_expansion} and \eqref{U_11_formula}, we have
$$
\left( \int_{\bS^{n-2}} a(\xi) |\bar{u}(r,\xi)|^2 \, d\xi \right)^{1/2} \ge |U_{1,j}(r)| \ge \frac{C_1}{2} r^\alpha \quad \mbox{for}~~ 0 < r < r_0,
$$
where $r_0$ is some small positive constant. Then, for any $r \in (0,r_0)$, there exists a $\xi_0(r) \in \bS^{n-2}$ such that
$$
|\bar{u}(r, \xi_0(r))| \ge \frac{1}{C_2} r^{\alpha}
$$
for some positive constant $C_2$. Since $\bar u$ is the average of $u$ in the $x_n$ direction, by \eqref{grad_u_bound_rough} with $\varepsilon = 0$, we have
$$
|u(r,\xi_0(r), 0) - \bar{u}(r, \xi_0(r))| \le C r^2 \sup_{x_n\in (g(x'),f(x'))} |\partial_{x_n} u(r,\xi_0(r),x_n)| \le Cr.
$$
Therefore, there exists a small constant $r_1$ such that for any $r \in (0, r_1)$,
\begin{equation*}
|u(r, \xi_0(r), 0)| \ge \frac{1}{2C_2} r^{\alpha}.
\end{equation*}
We denote $x_0 = (r , \xi_0(r) ,0)$. For a sufficiently large constant $C_3$, independent of $x_0$, we have, by Theorem \ref{touch_thm},
$$
\left|u\left( \frac{x_0}{C_3} \right)\right| \le C \left( \frac{|x_0'|}{C_3} \right)^\alpha \le \frac{1}{4C_2} |x_0'|^\alpha.
$$
Therefore, there exists an $x$ on the line segment between $x_0$ and $x_0/C_3$, such that
$$
|\nabla u(x)| \ge \frac{1}{C} |x'|^{\alpha-1}
$$
for some positive constant $C$ depending only on $n$, a positive lower bound of the eigenvalues of $D^2 (f-g)(0')$, and upper bounds of $\|\partial D_1 \|_{C^{4}}$ and $\|\partial D_2\|_{C^{4}}$. This concludes the proof.
\end{proof}

\section{The variable coefficients case}

In this section, we study the insulated problem with variable coefficients in dimension $n \ge 3$:
\begin{equation}
\label{variable_coefficient_2}
\left\{
\begin{aligned}
-\partial_i (A^{ij}(x)  \partial_j u(x)) &=0 \quad \mbox{in }\Omega_{R_0},\\
A^{ij}(x)  \partial_j u(x) \nu_i &= 0 \quad \mbox{on } \Gamma_+ \cup \Gamma_-,\\
\|u\|&_{L^\infty(\Omega_{R_0})} \le 1,
\end{aligned}
\right.
\end{equation}
where $\Gamma_+$ and $\Gamma_-$ are given in \eqref{Gamma_plusminus}, $(A^{ij}(x)) \in C^\gamma(\Omega_{R_0}), \gamma > 0$ is symmetric and uniformly elliptic with the Lipschitz constant $\sigma$, i.e.
$$
A^{ij}(x) = A^{ji}(x), \quad \sigma I \le A(x) \le \frac{1}{\sigma} I.
$$
We want to find a point $x_0$ and a linear transformation $l$, so that after the linear transformation, the coefficients $A^{ij}$ becomes $\delta_{ij}$ at the point $l(x_0)$, and $l(x_0)$ is the middle point of the closet points of $\widetilde{\Gamma}_+ := l (\Gamma_+)$ and $\widetilde{\Gamma}_- := l (\Gamma_-)$. Then we can apply Theorem \ref{touch_thm} or \ref{gen_thm} to get the gradient estimates.

When $\varepsilon = 0$, $x_0$ is the origin, and $l = C^{-1}(0)$, where $C(x) = \sqrt{A(x)}$. When $\varepsilon \neq 0$, by the change of variables
\begin{equation*}
\left\{
\begin{aligned}
y' &= x' ,\\
y_n &= x_n - g(x'),
\end{aligned}
\right.
\quad \forall (x',x_n) \in \Omega_{R_0},
\end{equation*}
we may assume that $g \equiv 0.$ Then any linear transformation $l$ (with no translation) maps the lower boundary $\Gamma_-$ to a hyperplane
$\widetilde{\Gamma}_-$. It also maps the tangent plane $x_n = \varepsilon/2$ of the upper boundary $\Gamma_+$ to the tangent plane of $\widetilde{\Gamma}_+$, which is paralleled to $\widetilde{\Gamma}_-$ as the mapping is linear. Then
$l(e_n\varepsilon/2)$ is the closest point on $\widetilde{\Gamma}_+$ to $\widetilde{\Gamma}_-$. Let $C(x) = \sqrt{A(x)}$ and $C_n$ be the last column of $C(x)$. We have the following Lemma.

\begin{lemma}
Under the settings above, let $R = \sqrt{n-1} \varepsilon / (2\sigma^2)$, there exists $x_0 \in \overline{B}_R \cap \{x_n = 0\}$ such that with the mapping $l = C^{-1}(x_0)$, $l(x_0)$ is the middle point of the closet points of $\widetilde{\Gamma}_+$ and $\widetilde{\Gamma}_-$.
\end{lemma}

\begin{proof}
It is easily seen that the normal direction of $\widetilde{\Gamma}_-$ is given by $C_n(x_0)$. By linearity, the distance from $l(x_0)$ to $l(\{x_n = \varepsilon/2\})$ is equal to the distance from $l(x_0)$ to $\widetilde{\Gamma}_-$. Thus, it suffices to have $l(x_0 - e_n \varepsilon/2) \| C_n(x_0)$. This is equivalent
to $x_0 - e_n \varepsilon/2 \| C(x_0) C_n(x_0)$, where $C(x_0) C_n(x_0) =: A_n(x_0)$ is the last column of $A(x_0)$. Thus, we only need to have
$$
(x_0)' = -\frac{\varepsilon}{2} \frac{(A_n(x_0))'}{A^{nn}(x_0)},
$$
where $(x_0)'$ and $(A_n(x_0))'$ are the first $n-1$ components of $x_0$ and $A_n(x_0)$, respectively. Now we define a mapping $T$ on $R^{n-1}$ by
$$
Ty = -\frac{\varepsilon}{2} \frac{(A_n(y,0))'}{A^{nn}(y,0)}.
$$
Clearly, $T$ is continuous. Since $A^{nn} \ge \sigma$ and $|A^{ni}| \le 1/ \sigma$, for $i = 1, 2, \ldots, n-1$, we have $|Ty| \le R$ for any $y \in \overline{B}_R$. By the Brouwer fixed point theorem, $T$ has a fixed point $(x_0)' \in \overline{B}_R$.
\end{proof}

After applying this linear transform and picking an appropriate coordinate system, we reduce the problem \eqref{variable_coefficient_2} to the case when $A^{ij}(0) = \delta_{ij}$. Therefore, Theorems \ref{touch_thm} and \ref{gen_thm} apply.

\bibliographystyle{amsplain}

\begin{bibdiv}
\begin{biblist}

\bib{ACKLY}{article}{
      author={Ammari, H.},
      author={Ciraolo, G.},
      author={Kang, H.},
      author={Lee, H.},
      author={Yun, K.},
       title={Spectral analysis of the {N}eumann-{P}oincar\'{e} operator and
  characterization of the stress concentration in anti-plane elasticity},
        date={2013},
        ISSN={0003-9527},
     journal={Arch. Ration. Mech. Anal.},
      volume={208},
      number={1},
       pages={275\ndash 304},
  url={https://doi-org.proxy.libraries.rutgers.edu/10.1007/s00205-012-0590-8},
      review={\MR{3021549}},
}

\bib{ADY}{article}{
      author={Ammari, H.},
      author={Davies, B.},
      author={Yu, S.},
       title={Close-to-touching acoustic subwavelength resonators:
  eigenfrequency separation and gradient blow-up},
        date={2020},
        ISSN={1540-3459},
     journal={Multiscale Model. Simul.},
      volume={18},
      number={3},
       pages={1299\ndash 1317},
         url={https://doi-org.proxy.libraries.rutgers.edu/10.1137/20M1313350},
      review={\MR{4128998}},
}

\bib{AKKY}{article}{
      author={Ammari, H.},
      author={Kang, H.},
      author={Kim, D.W.},
      author={Yu, S.},
       title={Quantitative estimates for stress concentration of the stokes
  flow between adjacent circular cylinders},
        date={2020},
        note={arXiv:2003.06578},
}

\bib{AKLLL}{article}{
      author={Ammari, H.},
      author={Kang, H.},
      author={Lee, H.},
      author={Lee, J.},
      author={Lim, M.},
       title={Optimal estimates for the electric field in two dimensions},
        date={2007},
        ISSN={0021-7824},
     journal={J. Math. Pures Appl. (9)},
      volume={88},
      number={4},
       pages={307\ndash 324},
  url={https://doi-org.proxy.libraries.rutgers.edu/10.1016/j.matpur.2007.07.005},
      review={\MR{2384571}},
}

\bib{AKL}{article}{
      author={Ammari, H.},
      author={Kang, H.},
      author={Lim, M.},
       title={Gradient estimates for solutions to the conductivity problem},
        date={2005},
        ISSN={0025-5831},
     journal={Math. Ann.},
      volume={332},
      number={2},
       pages={277\ndash 286},
  url={https://doi-org.proxy.libraries.rutgers.edu/10.1007/s00208-004-0626-y},
      review={\MR{2178063}},
}

\bib{BASL}{article}{
      author={Babu\v{s}ka, I.},
      author={Andersson, B.},
      author={Smith, P.J.},
      author={Levin, K.},
       title={Damage analysis of fiber composites. {I}. {S}tatistical analysis
  on fiber scale},
        date={1999},
        ISSN={0045-7825},
     journal={Comput. Methods Appl. Mech. Engrg.},
      volume={172},
      number={1-4},
       pages={27\ndash 77},
  url={https://doi-org.proxy.libraries.rutgers.edu/10.1016/S0045-7825(98)00225-4},
      review={\MR{1685902}},
}

\bib{BLY1}{article}{
      author={Bao, E.},
      author={Li, Y.Y.},
      author={Yin, B.},
       title={Gradient estimates for the perfect conductivity problem},
        date={2009},
        ISSN={0003-9527},
     journal={Arch. Ration. Mech. Anal.},
      volume={193},
      number={1},
       pages={195\ndash 226},
  url={https://doi-org.proxy.libraries.rutgers.edu/10.1007/s00205-008-0159-8},
      review={\MR{2506075}},
}

\bib{BLY2}{article}{
      author={Bao, E.},
      author={Li, Y.Y.},
      author={Yin, B.},
       title={Gradient estimates for the perfect and insulated conductivity
  problems with multiple inclusions},
        date={2010},
        ISSN={0360-5302},
     journal={Comm. Partial Differential Equations},
      volume={35},
      number={11},
       pages={1982\ndash 2006},
  url={https://doi-org.proxy.libraries.rutgers.edu/10.1080/03605300903564000},
      review={\MR{2754076}},
}

\bib{BLL}{article}{
      author={Bao, J.G.},
      author={Li, H.G.},
      author={Li, Y.Y.},
       title={Gradient estimates for solutions of the {L}am\'{e} system with
  partially infinite coefficients},
        date={2015},
        ISSN={0003-9527},
     journal={Arch. Ration. Mech. Anal.},
      volume={215},
      number={1},
       pages={307\ndash 351},
  url={https://doi-org.proxy.libraries.rutgers.edu/10.1007/s00205-014-0779-0},
      review={\MR{3296149}},
}

\bib{BLL2}{article}{
      author={Bao, J.G.},
      author={Li, H.G.},
      author={Li, Y.Y.},
       title={Gradient estimates for solutions of the {L}am\'{e} system with
  partially infinite coefficients in dimensions greater than two},
        date={2017},
        ISSN={0001-8708},
     journal={Adv. Math.},
      volume={305},
       pages={298\ndash 338},
  url={https://doi-org.proxy.libraries.rutgers.edu/10.1016/j.aim.2016.09.023},
      review={\MR{3570137}},
}

\bib{BT1}{incollection}{
      author={Bonnetier, E.},
      author={Triki, F.},
       title={Pointwise bounds on the gradient and the spectrum of the
  {N}eumann-{P}oincar\'{e} operator: the case of 2 discs},
        date={2012},
   booktitle={Multi-scale and high-contrast {PDE}: from modelling, to
  mathematical analysis, to inversion},
      series={Contemp. Math.},
      volume={577},
   publisher={Amer. Math. Soc., Providence, RI},
       pages={81\ndash 91},
  url={https://doi-org.proxy.libraries.rutgers.edu/10.1090/conm/577/11464},
      review={\MR{2985067}},
}

\bib{BT2}{article}{
      author={Bonnetier, E.},
      author={Triki, F.},
       title={On the spectrum of the {P}oincar\'{e} variational problem for two
  close-to-touching inclusions in 2{D}},
        date={2013},
        ISSN={0003-9527},
     journal={Arch. Ration. Mech. Anal.},
      volume={209},
      number={2},
       pages={541\ndash 567},
  url={https://doi-org.proxy.libraries.rutgers.edu/10.1007/s00205-013-0636-6},
      review={\MR{3056617}},
}

\bib{BV}{article}{
      author={Bonnetier, E.},
      author={Vogelius, M.},
       title={An elliptic regularity result for a composite medium with
  ``touching'' fibers of circular cross-section},
        date={2000},
        ISSN={0036-1410},
     journal={SIAM J. Math. Anal.},
      volume={31},
      number={3},
       pages={651\ndash 677},
  url={https://doi-org.proxy.libraries.rutgers.edu/10.1137/S0036141098333980},
      review={\MR{1745481}},
}

\bib{BudCar}{article}{
      author={Budiansky, B.},
      author={Carrier, G.~F.},
       title={{High Shear Stresses in Stiff-Fiber Composites}},
        date={1984},
        ISSN={0021-8936},
     journal={Journal of Applied Mechanics},
      volume={51},
      number={4},
       pages={733\ndash 735},
         url={https://doi.org/10.1115/1.3167717},
}

\bib{CKN}{article}{
      author={Caffarelli, L.},
      author={Kohn, R.},
      author={Nirenberg, L.},
       title={First order interpolation inequalities with weights},
        date={1984},
        ISSN={0010-437X},
     journal={Compositio Math.},
      volume={53},
      number={3},
       pages={259\ndash 275},
         url={http://www.numdam.org/item?id=CM_1984__53_3_259_0},
      review={\MR{768824}},
}

\bib{CY}{article}{
      author={Capdeboscq, Y.},
      author={Yang~Ong, S.C.},
       title={Quantitative {J}acobian determinant bounds for the conductivity
  equation in high contrast composite media},
        date={2020},
        ISSN={1531-3492},
     journal={Discrete Contin. Dyn. Syst. Ser. B},
      volume={25},
      number={10},
       pages={3857\ndash 3887},
  url={https://doi-org.proxy.libraries.rutgers.edu/10.3934/dcdsb.2020228},
      review={\MR{4147367}},
}

\bib{CodLev}{book}{
      author={Coddington, E.},
      author={Levinson, N.},
       title={Theory of ordinary differential equations},
   publisher={McGraw-Hill Book Company, Inc., New York-Toronto-London},
        date={1955},
      review={\MR{0069338}},
}

\bib{DL}{article}{
      author={Dong, H.},
      author={Li, H.G.},
       title={Optimal estimates for the conductivity problem by {G}reen's
  function method},
        date={2019},
        ISSN={0003-9527},
     journal={Arch. Ration. Mech. Anal.},
      volume={231},
      number={3},
       pages={1427\ndash 1453},
  url={https://doi-org.proxy.libraries.rutgers.edu/10.1007/s00205-018-1301-x},
      review={\MR{3902466}},
}

\bib{DLY}{article}{
      author={Dong, H.},
      author={Li, Y.Y.},
      author={Yang, Z.},
       title={Optimal gradient estimates of solutions to the insulated
  conductivity problem in dimension greater than two},
        date={2021},
        note={arXiv:2110.11313},
}

\bib{DZ}{article}{
      author={Dong, H.},
      author={Zhang, H.},
       title={On an elliptic equation arising from composite materials},
        date={2016},
        ISSN={0003-9527},
     journal={Arch. Ration. Mech. Anal.},
      volume={222},
      number={1},
       pages={47\ndash 89},
  url={https://doi-org.proxy.libraries.rutgers.edu/10.1007/s00205-016-0996-9},
      review={\MR{3519966}},
}

\bib{FKS}{article}{
      author={Fabes, E.B.},
      author={Kenig, C.E.},
      author={Serapioni, R.P.},
       title={The local regularity of solutions of degenerate elliptic
  equations},
        date={1982},
        ISSN={0360-5302},
     journal={Comm. Partial Differential Equations},
      volume={7},
      number={1},
       pages={77\ndash 116},
         url={https://doi.org/10.1080/03605308208820218},
      review={\MR{643158}},
}

\bib{GiaMar}{book}{
      author={Giaquinta, M.},
      author={Martinazzi, L.},
       title={An introduction to the regularity theory for elliptic systems,
  harmonic maps and minimal graphs},
     edition={Second},
      series={Appunti. Scuola Normale Superiore di Pisa (Nuova Serie) [Lecture
  Notes. Scuola Normale Superiore di Pisa (New Series)]},
   publisher={Edizioni della Normale, Pisa},
        date={2012},
      volume={11},
        ISBN={978-88-7642-442-7; 978-88-7642-443-4},
         url={https://doi.org/10.1007/978-88-7642-443-4},
      review={\MR{3099262}},
}

\bib{Gor}{article}{
      author={Gorb, Y.},
       title={Singular behavior of electric field of high-contrast concentrated
  composites},
        date={2015},
        ISSN={1540-3459},
     journal={Multiscale Model. Simul.},
      volume={13},
      number={4},
       pages={1312\ndash 1326},
         url={https://doi-org.proxy.libraries.rutgers.edu/10.1137/140982076},
      review={\MR{3418221}},
}

\bib{JiKang}{article}{
      author={Ji, Y.},
      author={Kang, H.},
       title={Spectrum of the neumann-poincar\'e operator and optimal estimates
  for transmission problems in presence of two circular inclusions},
        date={2021},
        note={arXiv:2105.06093},
}

\bib{Kang}{inproceedings}{
      author={Kang, H.},
       title={Quantitative analysis of field concentration in presence of
  closely located inclusions of high contrast},
   booktitle={Proceedings of the {I}nternational {C}ongress of {M}athematicians
  2022, to appear},
}

\bib{KLY1}{article}{
      author={Kang, H.},
      author={Lim, M.},
      author={Yun, K.},
       title={Asymptotics and computation of the solution to the conductivity
  equation in the presence of adjacent inclusions with extreme conductivities},
        date={2013},
        ISSN={0021-7824},
     journal={J. Math. Pures Appl. (9)},
      volume={99},
      number={2},
       pages={234\ndash 249},
  url={https://doi-org.proxy.libraries.rutgers.edu/10.1016/j.matpur.2012.06.013},
      review={\MR{3007847}},
}

\bib{KLY2}{article}{
      author={Kang, H.},
      author={Lim, M.},
      author={Yun, K.},
       title={Characterization of the electric field concentration between two
  adjacent spherical perfect conductors},
        date={2014},
        ISSN={0036-1399},
     journal={SIAM J. Appl. Math.},
      volume={74},
      number={1},
       pages={125\ndash 146},
         url={https://doi-org.proxy.libraries.rutgers.edu/10.1137/130922434},
      review={\MR{3162415}},
}

\bib{Kel}{article}{
      author={Keller, J.~B.},
       title={{Stresses in Narrow Regions}},
        date={1993},
        ISSN={0021-8936},
     journal={Journal of Applied Mechanics},
      volume={60},
      number={4},
       pages={1054\ndash 1056},
         url={https://doi.org/10.1115/1.2900977},
}

\bib{KL}{article}{
      author={Kim, J.},
      author={Lim, M.},
       title={Electric field concentration in the presence of an inclusion with
  eccentric core-shell geometry},
        date={2019},
        ISSN={0025-5831},
     journal={Math. Ann.},
      volume={373},
      number={1-2},
       pages={517\ndash 551},
  url={https://doi-org.proxy.libraries.rutgers.edu/10.1007/s00208-018-1688-6},
      review={\MR{3968879}},
}

\bib{L}{article}{
      author={Li, H.G.},
       title={Asymptotics for the {E}lectric {F}ield {C}oncentration in the
  {P}erfect {C}onductivity {P}roblem},
        date={2020},
        ISSN={0036-1410},
     journal={SIAM J. Math. Anal.},
      volume={52},
      number={4},
       pages={3350\ndash 3375},
         url={https://doi-org.proxy.libraries.rutgers.edu/10.1137/19M1282623},
      review={\MR{4126320}},
}

\bib{LLY}{article}{
      author={Li, H.G.},
      author={Li, Y.Y.},
      author={Yang, Z.},
       title={Asymptotics of the gradient of solutions to the perfect
  conductivity problem},
        date={2019},
        ISSN={1540-3459},
     journal={Multiscale Model. Simul.},
      volume={17},
      number={3},
       pages={899\ndash 925},
         url={https://doi-org.proxy.libraries.rutgers.edu/10.1137/18M1214329},
      review={\MR{3977105}},
}

\bib{LWX}{article}{
      author={Li, H.G.},
      author={Wang, F.},
      author={Xu, L.},
       title={Characterization of electric fields between two spherical perfect
  conductors with general radii in 3{D}},
        date={2019},
        ISSN={0022-0396},
     journal={J. Differential Equations},
      volume={267},
      number={11},
       pages={6644\ndash 6690},
  url={https://doi-org.proxy.libraries.rutgers.edu/10.1016/j.jde.2019.07.007},
      review={\MR{4001067}},
}

\bib{LN}{article}{
      author={Li, Y.Y.},
      author={Nirenberg, L.},
       title={Estimates for elliptic systems from composite material},
        date={2003},
        ISSN={0010-3640},
     journal={Comm. Pure Appl. Math.},
      volume={56},
      number={7},
       pages={892\ndash 925},
         url={https://doi-org.proxy.libraries.rutgers.edu/10.1002/cpa.10079},
      review={\MR{1990481}},
}

\bib{LV}{article}{
      author={Li, Y.Y.},
      author={Vogelius, M.},
       title={Gradient estimates for solutions to divergence form elliptic
  equations with discontinuous coefficients},
        date={2000},
        ISSN={0003-9527},
     journal={Arch. Ration. Mech. Anal.},
      volume={153},
      number={2},
       pages={91\ndash 151},
  url={https://doi-org.proxy.libraries.rutgers.edu/10.1007/s002050000082},
      review={\MR{1770682}},
}

\bib{LY2}{article}{
      author={Li, Y.Y.},
      author={Yang, Z.},
       title={Gradient estimates of solutions to the insulated conductivity
  problem in dimension greater than two},
        date={2020},
        note={arXiv:2012.14056, Math. Ann., to appear},
}

\bib{LimYun}{article}{
      author={Lim, M.},
      author={Yun, K.},
       title={Blow-up of electric fields between closely spaced spherical
  perfect conductors},
        date={2009},
        ISSN={0360-5302},
     journal={Comm. Partial Differential Equations},
      volume={34},
      number={10-12},
       pages={1287\ndash 1315},
  url={https://doi-org.proxy.libraries.rutgers.edu/10.1080/03605300903079579},
      review={\MR{2581974}},
}

\bib{Mar}{article}{
      author={Markenscoff, X.},
       title={Stress amplification in vanishingly small geometries},
        date={1996},
        ISSN={1432-0924},
     journal={Computational Mechanics},
      volume={19},
      number={1},
       pages={77\ndash 83},
         url={https://doi.org/10.1007/BF02824846},
}

\bib{We}{article}{
      author={Weinkove, B.},
       title={The insulated conductivity problem, effective gradient estimates
  and the maximum principle},
        date={2021},
        note={arXiv:2103.14143, Math. Ann., to appear},
}

\bib{Y1}{article}{
      author={Yun, K.},
       title={Estimates for electric fields blown up between closely adjacent
  conductors with arbitrary shape},
        date={2007},
        ISSN={0036-1399},
     journal={SIAM J. Appl. Math.},
      volume={67},
      number={3},
       pages={714\ndash 730},
         url={https://doi-org.proxy.libraries.rutgers.edu/10.1137/060648817},
      review={\MR{2300307}},
}

\bib{Y2}{article}{
      author={Yun, K.},
       title={Optimal bound on high stresses occurring between stiff fibers
  with arbitrary shaped cross-sections},
        date={2009},
        ISSN={0022-247X},
     journal={J. Math. Anal. Appl.},
      volume={350},
      number={1},
       pages={306\ndash 312},
  url={https://doi-org.proxy.libraries.rutgers.edu/10.1016/j.jmaa.2008.09.057},
      review={\MR{2476915}},
}

\bib{Y3}{article}{
      author={Yun, K.},
       title={An optimal estimate for electric fields on the shortest line
  segment between two spherical insulators in three dimensions},
        date={2016},
        ISSN={0022-0396},
     journal={J. Differential Equations},
      volume={261},
      number={1},
       pages={148\ndash 188},
  url={https://doi-org.proxy.libraries.rutgers.edu/10.1016/j.jde.2016.03.005},
      review={\MR{3487255}},
}

\end{biblist}
\end{bibdiv}

\end{document}